\documentclass[reqno,10pt]{amsart}

\usepackage[utf8]{inputenc}

\usepackage{amssymb,amsmath,amsthm,amsfonts,color}
\usepackage{mathrsfs,dsfont,comment,mathscinet}
\usepackage{todonotes,cancel,soul}
\usepackage{tikz}

\usepackage{enumerate,esint}
\usepackage[left=2.8cm,right=2.8cm,top=2.8cm,bottom=2.8cm]{geometry}
\usepackage{mathtools}
\usepackage{graphicx,tikz}
\usepackage[format=hang,labelfont=bf]{caption}
\usepackage{ wasysym }

\usepackage{epstopdf}
\usepackage[colorlinks=true, pdfstartview=FitV, linkcolor=blue, 
            citecolor=blue, urlcolor=blue]{hyperref}

\allowdisplaybreaks
\numberwithin{equation}{section}
\theoremstyle{plain}
\newtheorem{theorem}{Theorem}[section]
\newtheorem{proposition}[theorem]{Proposition}
\newtheorem{lemma}[theorem]{Lemma}
\newtheorem{example}[theorem]{Example}

\theoremstyle{definition}

\newtheorem{remark}[theorem]{Remark}

\newtheorem*{theorem*}{Theorem}

\makeatletter
\def\th@plain{%
  \thm@notefont{}
  \itshape 
}
\def\th@definition{%
  \thm@notefont{}
  \normalfont 
}
\makeatother

\definecolor{mblue}{HTML}{13439b}

\newcommand\R{\mathbb R}

\newcommand\N{\mathbb N}
\newcommand\Z{\mathbb Z}
\renewcommand\S{\mathbb S}

\renewcommand{\d}{\mathrm{d}}

\newcommand{\restr}[1]{|_{#1}}

\newcommand{\loc}{\mathrm{loc}}
\newcommand{\var}{\mathrm{Var}}

\newcommand{\cof}{\mathrm{cof}}
\newcommand{\adj}{\mathrm{adj}}

\renewcommand{\deg}{\mathrm{deg}}
\renewcommand{\div}{\mathrm{div}}
\newcommand{\curl}{\mathrm{curl}}
\newcommand{\dist}{\mathrm{dist}}
\newcommand\wk{\rightharpoonup}
\newcommand{\rtt}{\R^{3\times 3}}
\newcommand{\wks}{\overset{\ast}{\rightharpoonup}}

\newcommand*\closure[1]{\overline{#1}}

\newcommand{\leb}{\mathscr{L}^3}
\newcommand{\haus}{\mathscr{H}^2}

\newcommand{\mres}{\mathbin{\vrule height 1.6ex depth 0pt width
		0.13ex\vrule height 0.13ex depth 0pt width 1.3ex}}

\def\XXint#1#2#3{{\setbox0=\hbox{$#1{#2#3}{\int}$ }
\vcenter{\hbox{$#2#3$ }}\kern-.6\wd0}}

\newcommand\RRR{\color{black}}
\newcommand\EEE{\color{black}}
\newcommand\MMM{\color{black}}

\title{
\MMM Existence results in large-strain magnetoelasticity}
\author[M. Bresciani]{Marco Bresciani}
\address[M. Bresciani]{
		Institute of Analysis and Scientific Computing,
		TU Wien,
		Wiedner Hauptstrasse 8-10, 1040 Vienna, Austria.
		Email: \href{mailto:marco.bresciani@tuwien.ac.at}{\tt marco.bresciani@tuwien.ac.at}.
}
\author[E. Davoli]{Elisa Davoli}
\address[E. Davoli]{
		Institute of Analysis and Scientific Computing,
		TU Wien,
		Wiedner Hauptstrasse 8-10, 1040 Vienna, Austria.
		Email: \href{mailto:elisa.davoli@tuwien.ac.at}{\tt elisa.davoli@tuwien.ac.at}.
}
\author[M. Kru\v{z}\'ik]{Martin Kru\v{z}\'ik}
\address[M. Kru\v{z}\'ik]{Czech Academy of Sciences, 
Institute of Information Theory and Automation, 
Pod Vod\'arenskou V\v{e}\v{z}\'i 4, 182 00 Prague, Czechia and Faculty of Civil Engineering, Czech Technical
University, Th\'{a}kurova 7, 166 29 Prague, Czechia.
		Email: \href{mailto: kruzik@utia.cas.cz}{\tt kruzik@utia.cas.cz}.
}

\date{\today}
\keywords{magnetoelasticity, Eulerian-{\MMM Lagrangian} energies, quasistatic evolution, large-strain theories}
\subjclass[2000]{49J45; 74C99; 74F15}

\begin{document}

\setlength\parindent{0pt}

\vskip .2truecm
\begin{abstract}
We investigate  \MMM variational \EEE problems in large-strain  magnetoelasticity, \MMM both in the static and in the quasistatic setting\EEE. \MMM The model contemplates a mixed Eulerian-{\MMM Lagrangian} formulation: while deformations are defined on the reference configuration, magnetizations are defined on the deformed set in the actual space. In the static setting, we establish the existence of minimizers. In particular, we provide a compactness result for sequences of admissible states with equi-bounded energies which gives the convergence of the composition of magnetizations with deformations. In the quasistatic setting, we consider a notion of dissipation which is  frame-indifferent and we show that the incremental minimization problem is solvable. Then, we propose a regularization of the model in the spirit of gradient polyconvexity and we prove the existence of energetic solutions for the regularized model.\EEE 
\end{abstract}
\maketitle

\section{Introduction}

In this paper, {\MMM we study a variational model for magnetoelastic materials at large strains and we provide existence results for optimal configurations, both in the static and in the quasistatic setting.}

The theory of {\MMM Brown \cite{brown}} (see also \cite{desimone.dolzmann,desimone.podioguidugli,james.kinderlehrer}) is based on the assumption that equilibrium configurations of a magnetoelastic body are {\MMM given by} minimizers of  an energy functional that depends on the deformation $\boldsymbol{y}:\Omega\to\R^3$ and on the magnetization  $\boldsymbol{m}:\boldsymbol{y}(\Omega)\to\S^2$, where $\Omega \subset \R^3$ represents the reference configuration of the body and $\S^2$ denotes the unit sphere in $\R^3$. {\MMM The fact that magnetizations are sphere-valued resembles the constraint of magnetic saturation \cite{brown}  which, up to normalization, reads $|\boldsymbol{m}|=1$ in $\boldsymbol{y}(\Omega)$.} 
The {\MMM magnetoelastic} energy functional is defined, for $\boldsymbol{q}=(\boldsymbol{y},\boldsymbol{m})$, by setting
\begin{equation}
    \label{eqn:energy-E-0}
    E(\boldsymbol{q})\coloneqq \int_\Omega W(\nabla\boldsymbol{y},\boldsymbol{m}\circ \boldsymbol{y})\,\d\boldsymbol{x}+  \alpha\int_{\boldsymbol{y}(\Omega)}|\nabla \boldsymbol{m}|^2\,\d \boldsymbol{\xi}+\frac{\mu_0}{2}\int_{\R^3}|\nabla \zeta_{\boldsymbol{m}}|^2\,\d\boldsymbol{\xi}. 
\end{equation}
Here, $W$ denotes a nonlinear, frame-indifferent, magnetostrictive energy density {\MMM and the corresponding integral in \eqref{eqn:energy-E-0} represents the elastic energy}.  The second term in \eqref{eqn:energy-E-0} is termed exchange energy, penalizing spatial changes of $\boldsymbol{m}$; $\alpha>0$ is the exchange constant. The third contribution in \eqref{eqn:energy-E-0} encodes the magnetostatic energy and favors  divergence-free states of the magnetization; $\mu_0>0$ is the permeability of the vacuum.  In particular, the stray-field potential  $\zeta_{\boldsymbol{m}}\colon \R^3 \to \R$ is defined as a weak solution of the magnetostatic Maxwell equation \cite{brown,desimone.podioguidugli,james.kinderlehrer}:
\begin{equation*}
    \Delta \zeta_{\boldsymbol{m}}=\div(\chi_{\boldsymbol{y}(\Omega)}\boldsymbol{m})\:\text{in}\:\R^3.
\end{equation*}

{\MMM The magnetoelastic energy functional usually includes additional terms, such as the anisotropy energy and the asymmetric exchange energy, also known as Dzyaloshinskii-Moriya Interaction (DMI) energy \cite{dzyaloshinsky,moriya}. For $\boldsymbol{q}=(\boldsymbol{y},\boldsymbol{m})$, these two terms are respectively given by
\begin{equation}
\label{eqn:energy-anisotropy-DMI}
 E^{\rm ani}(\boldsymbol{q})\coloneqq \int_{\boldsymbol{y}(\Omega)}\phi(\boldsymbol{m})\,\d\boldsymbol{\xi},\qquad E^{\rm DMI}(\boldsymbol{q})\coloneqq \kappa \int_{\boldsymbol{y}(\Omega)} \curl\boldsymbol{m}\cdot\boldsymbol{m}\,\d\boldsymbol{\xi}.
\end{equation}
The first term takes into account the magnetocrystalline anisotropy; the function $\phi\colon \S^2\to\R$ is  continuous and nonnegative, and vanishes only on a finite set of directions, the easy axis, along which magnetizations  tend  to align themselves spontaneously. The DMI energy is linked with the possible lack of centrosymmetry in the crystalline structure.} The sign of the parameter $\kappa\in \mathbb{R}$ is not prescribed. According to its value, this energy term alone would be minimized by configurations satisfying $\curl\, \boldsymbol{m}=\pm \boldsymbol{m}$, or equivalently, $\pm\boldsymbol{m}=-\Delta \boldsymbol{m}$. {\MMM Altogether}, the sum of the symmetric and DMI exchange is optimized by helical fields $\boldsymbol{m}$ describing a rotation of constant frequency $\kappa$ orthogonal to one of the coordinate axes, and rotating clockwise or counter-clockwise according to the sign of $\kappa$ \cite{davoli.difratta,lund.muratov,melcher}.

{\MMM The two energy contributions in \eqref{eqn:energy-anisotropy-DMI} do not introduce additional mathematical challenges in our analysis as the corresponding functionals are continuous with respect to the topology considered. Therefore, these energy terms will be neglected and we will consider the magnetoelastic energy functional in \eqref{eqn:energy-E-0}. The same holds for applied loads which, in first instance, will also be neglected. These comprise applied body and surfaces forces and external magnetic fields.}

A peculiar feature of the energy in \eqref{eqn:energy-E-0} consists in its mixed Eulerian-{\MMM Lagrangian} structure. Whereas the elastic energy is evaluated on the reference configuration, and is hence {\MMM Lagrangian}, in fact, all magnetic contributions are set on the actual deformed set, thus being Eulerian. 
{\MMM Therefore,} the problem needs to be formulated in a class of admissible states such that the deformed set $\boldsymbol{y}(\Omega)$ {\MMM corresponding to each deformation} $\boldsymbol{y}$ {\MMM can be suitably interpreted}.

The existence of minimizers for the functional in \eqref{eqn:energy-E-0} has been {\MMM firstly} proven in \cite{rybka.luskin} for non-simple materials, and then in \cite{barchiesi.desimone,kruzik.stefanelli.zeman} for simple materials under the constraint of incompressibility. {\MMM In \cite{kruzik.stefanelli.zeman}, the authors also studied quasistatic evolutions driven by time-dependent applied loads and a rate-independent dissipation, and established the existence of energetic solutions.} Subsequently, in \cite{barchiesi.henao.moracorral}, the existence of minimizers for \eqref{eqn:energy-E-0} has been obtained for compressible materials under weak growth assumptions on the elastic energy density. The analysis in this case becomes quite technical because of the  possible discontinuity of deformations. {\MMM A further extension of this result, contemplating an even larger class of admissible deformations, has been obtained in \cite{henao.stroffolini}. Actually, the analysis in \cite{barchiesi.desimone,henao.stroffolini} concerns nematic elastomers; in that case, the variable $\boldsymbol{m}$ represents the nematic director and the magnetostatic energy term is dropped. However, from the mathematical point of view, the problem is substantially the same as the one of magnetoelasticity.}

To complete our review on magnetoelasticity, we also mention a few recent works dealing with the analysis of magnetoelastic thin films. In particular, in \cite{kruzik.stefanelli.zanini} magnetoelastic plates and their corresponding quasistatic evolutions are studied within the  framework of linearized elasticity {\MMM in a purely {\MMM Lagrangian} setting}. A large-strain analysis of magnetoelastic plates has been initiated in \cite{liakhova}, under a priori constraints on the Jacobian {\MMM determinant} of deformations (see also \cite{liakhova.luskin.zhang,luskin.zhang} for numerical results).  The membrane regime for nonsimple materials has been recently investigated in \cite{davoli.kruzik.piovano.stefanelli}, whereas von K\'{a}rm\'{a}n theories  starting from a nonlinear model have been identified in \cite{bresciani} in the case of incompressible materials.

{\MMM This paper is subdivided into two parts. The first part concerns the analysis of the variational model in the static case. Our setting is comparable with the one in \cite{kruzik.stefanelli.zeman}, except that, here, we are not restricted to incompressible materials. In particular, the coercivity properties of the elastic energy density ensure that admissible deformations are continuous.
The main contribution in the first part consists in proving existence of equilibrium configurations for the magnetoelastic energy in \eqref{eqn:energy-E-0}.}
A simplified version of \MMM this \EEE result reads as follows. We refer to Theorem \ref{thm:existence-minimizers} for the precise statement and assumptions.

\begin{theorem}[\MMM Existence of minimizers\EEE]
\label{thm:meta1}
Assume that the {\MMM elastic} energy density $W$ is continuous, $p$-coercive {\MMM with} $p>3$, blows up under extreme compressions and is polyconvex {\MMM in its first argument}. Assume also that interpenetration of matter is prevented. Then, the magnetoelastic energy in \eqref{eqn:energy-E-0} admits a minimizer.
\end{theorem}

{\MMM As already mentioned, the existence of minimizers for the functional in \eqref{eqn:energy-E-0} has been already established, even for larger classes of admissible deformations, in \cite{barchiesi.desimone,barchiesi.henao.moracorral,henao.stroffolini}. However,  here, the result is proven in a more direct way by arguing similarly to \cite{kruzik.stefanelli.zeman}. The main point is the compactness of sequences of admissible states with equi-bounded energies achieved in Proposition \ref{prop:compactness}. In particular, we prove the convergence of the composition of magnetizations with deformations. From this, the lower semicontinuity of the elastic energy, which represents the most problematic term, is obtained by a standard application of the classical Eisen Selection Lemma \cite{eisen}. 

The convergence of the compositions of magnetizations with deformations constitutes one of the main novelties of our paper and is going to be fundamental for the analysis in the quasistatic setting. Note that this is a very delicate issue: indeed, contrary to deformations, magnetizations may be discontinuous.
In \cite{kruzik.stefanelli.zeman}, the convergence of compositions follows easily from the fact that deformations are volume-preserving. Instead, in \cite{barchiesi.henao.moracorral}, this issue is circumvented by working in the deformed configuration and by exploiting the weak convergence of inverse deformations together with their Jacobian minors.

The techniques employed in our analysis require a careful study of the geometry of the deformed set and of fine and invertibility properties of admissible deformations. Essential tools are given by the topological degree \cite{fonseca.gangbo} and by refined versions of the area formula and the change-of-variable formula \cite{federer,hajlasz}. 
In Lemma \ref{lem:deformed-configuration}, we see that the deformed set $\boldsymbol{y}(\Omega)$ can be replaced with a suitable open subset of it which has full measure. This is necessary in order to be able to give a precise meaning to the gradient of magnetizations appearing in the expression of the exchange energy. 

The proof of the convergence of the compositions of magnetizations with deformations combines three main ingredients: the convergences of the deformed sets, the equi-integrability of the Jacobian determinants of inverse deformations, and two classical results of measure theory, namely  Egorov and Lusin Theorems.} For a similar approach relying on the equi-integrability of Jacobian determinants of inverse {\MMM deformations}, we refer to \cite{grandi.kruzik.mainini.stefanelli}.
{\MMM In that paper, the desired equi-integrability property follows from some a priori control on the distortion of admissible deformations obtained by imposing specific growth conditions on the elastic energy density, while, here, this property is deduced from the singular behaviour of the elastic energy density in response to extreme compressions. This allows us to work with a more natural class of admissible deformations, which are not necessarily homeomorphisms.

}

{\MMM In} the second part of our paper,  {\MMM we} study quasi-static evolutions driven by the energy functional in \eqref{eqn:energy-E-0}, complemented by {\MMM the work of }time-dependent applied loads determined by external body forces, surface forces, and magnetic fields, {\MMM and a rate-independent dissipation}. Our analysis is set within the theory of rate-independent processes \cite{mielke.roubicek} {  \MMM with the notion of} energetic solution. 

{\MMM Our setting is again similar to the one in \cite{kruzik.stefanelli.zeman},  but} a key difference consists in our definition of dissipation distance. {\MMM In \cite{kruzik.stefanelli.zeman}, this is simply defined as the distance in $L^1$ between the compositions of magnetizations with deformations. Here, instead, this is constructed by introducing a dissipative variable, the Lagrangian magnetization, which is obtained as the  pull-back of the magnetization to the reference configuration. More precisely,} for $\boldsymbol{q}=(\boldsymbol{y},\boldsymbol{m})\in \mathcal{Q}$, this is given by
\begin{equation}
    \label{eqn:Lagrangean-magnetization-intro}
    \mathcal{Z}(\boldsymbol{q})\coloneqq (\adj \nabla \boldsymbol{y}) \, \boldsymbol{m}\circ \boldsymbol{y},
\end{equation}
where {\MMM the adjugate matrix denotes simply} the transpose of the cofactor matrix.
Then, the dissipation distance $\mathcal{D}\colon \mathcal{Q}\times \mathcal{Q}\to [0,+\infty)$ is defined as
\begin{equation}
    \label{eqn:dissipation-distance-intro}
    \mathcal{D}(\boldsymbol{q},\widehat{\boldsymbol{q}})\coloneqq \int_\Omega |\mathcal{Z}(\boldsymbol{q})-\mathcal{Z}(\widehat{\boldsymbol{q}})|\,\d \boldsymbol{x}.
\end{equation}
{\MMM We observe that this dissipation is frame-indifferent, i.e. rigid motions do not dissipate energy. This because the Lagrangian magnetization in \eqref{eqn:Lagrangean-magnetization-intro} is an objective quantity. To see this, let $\boldsymbol{T}\colon \R^3\to \R^3$ be a rigid motion of the form $\boldsymbol{T}(\boldsymbol{\xi})\coloneqq \boldsymbol{Q}\,\boldsymbol{\xi}+\boldsymbol{c}$ for every $\boldsymbol{\xi}\in \R^3$, where $\boldsymbol{Q}\in SO(3)$ and $\boldsymbol{c}\in \R^3$. The admissible state $\widetilde{\boldsymbol{q}}=(\widetilde{\boldsymbol{y}},\widetilde{\boldsymbol{m}})\in\mathcal{Q}$ obtained from $\boldsymbol{q}=(\boldsymbol{y},\boldsymbol{m})\in \mathcal{Q}$ by superposition with $\boldsymbol{T}$ is defined by setting $\widetilde{\boldsymbol{y}}\coloneqq \boldsymbol{T}\circ \boldsymbol{y}$ and $\widetilde{\boldsymbol{m}}\coloneqq \boldsymbol{Q}\,(\boldsymbol{m}\circ \boldsymbol{T}^{-1})$. Thus, $\mathcal{Z}(\widetilde{\boldsymbol{q}})=\mathcal{Z}(\boldsymbol{q})$.}
{\MMM This observation demonstrates that the dissipation introduced in \eqref{eqn:dissipation-distance-intro} is selective enough to only account for the part of the magnetic reorientation that corresponds to changes of the geometry of the deformed sets.}

{\MMM
We mention that \eqref{eqn:Lagrangean-magnetization-intro} is not the only possible way to pull-back of the magnetization to the reference configuration as an objective quantity. Our choice preserves the flux of magnetizations though closed surfaces. Another possible choice, preserving the circulation of magnetizations along closed loops, consists in replacing the matrix $\adj\nabla\boldsymbol{y}$ in \eqref{eqn:Lagrangean-magnetization-intro} with $(\nabla\boldsymbol{y})^\top$. Both quantities have been considered in different contexts (see \cite{antman.rogers,rogers,roubicek.tomassetti} for the flux-preserving pull-back and \cite{kankanala.triantafyllidis,sharma.saxena} for the circulation-preserving pull-back). Here, the pull-back in \eqref{eqn:Lagrangean-magnetization-intro} is preferred as it appears naturally while rewriting the magnetostatic energy as an integral on the reference configuration \cite{rogers}.
}
\EEE

{\MMM The compactness established in Proposition \ref{prop:compactness}, ensures that the dissipation distance in \eqref{eqn:dissipation-distance-intro} is lower semicontinuous on the sublevel sets of the total energy. Specifically, this follows combining the weak continuity of the Jacobian cofactor with the convergence of the compositions of magnetizations with deformations.
As a consequence, the incremental minimization problem is solvable for each fixed partition of the time interval. 
Nevertheless, the existence of energetic solutions is} out-of-reach in {\MMM this} framework. {\MMM Roughly speaking, this is because the dissipation distance is not continuous on the sublevel sets of the total energy. Such a situation is quite common for large-strain theories (see \cite{mainik.mielke} for an example in finite plasticity).}

{\MMM Therefore, in the last part of the paper,} we resort to a regularized counterpart to the functional in \eqref{eqn:energy-E-0}, which is obtained by augmenting the magnetoelastic energy by the total variation of the Jacobian cofactor of the deformation. Namely, for every $\boldsymbol{q}=(\boldsymbol{y},\boldsymbol{m})\in \mathcal{Q}$, we {\MMM consider the regularized energy functional}
\begin{equation}
\label{eq:int-en-reg}
    \widetilde{E}(\boldsymbol{q})\coloneqq E(\boldsymbol{q})+|D (\cof\,\nabla \boldsymbol{y})|(\Omega).
\end{equation}
This brings us to the theory of nonsimple materials initiated by Toupin \cite{toupin1, toupin2} and later extended by many authors (see, for instance, \cite{ball.currie.olver,mielke.roubicek,roubicek.tomassetti}). The idea is to assume that the stored energy density depends also on higher-order gradients of the deformation.  More regularity allows us to work in a stronger topology and {\MMM to gain the continuity of the dissipation distance on the sublevel sets of the regularized total energy}.  Here, we apply a fairly weak concept of nonsimple materials  introduced in \cite{benesova.kruzik.schloemerkemper} under the name of { gradient polyconvex materials} (see also \cite{kruzik.roubicek}). Indeed, in view of \eqref{eq:int-en-reg},  we only need to assume that $\cof\,\nabla \boldsymbol{y}\in BV(\Omega;\R^{3\times 3})$. 

{\MMM Our second main result asserts the existence of energetic solutions for the regularized model. We present a simplified statement below and} we refer to Theorem \ref{thm:existence-energetic-solution} for its precise formulation.

\begin{theorem}[\MMM Existence of energetic solutions\EEE]
\label{thm:meta2}
{\MMM Under the same assumptions of Theorem \ref{thm:meta1}, there exists an energetic solution for the regularized model determined by the energy in \eqref{eq:int-en-reg}, complemented by time-dependent applied loads,} and the dissipation in \eqref{eqn:dissipation-distance-intro}.
\end{theorem}

{\MMM The existence of energetic solutions is proved by time-discretization following the well-established scheme introduced in \cite{francfort.mielke} (see also \cite{mielke.roubicek}). Thus, the compactness of time-discrete solutions is achieved by appealing to some versione of the Helly Selection Theorem, here provided by Lemma \ref{lem:helly}. We stress that the existence of time-discrete solutions is 
available even in the absence of the regularization introduced in \eqref{eq:int-en-reg} (see Proposition \ref{prop:solutions-imp}), and that this is only needed to construct time-continuous solutions. 

Note that taking the same definition of dissipation distance  as in \cite{kruzik.stefanelli.zeman}, simply given by the distance in $L^1$ between the compositions of magnetizations with deformations, we would be able to establish the existence of energetic solutions without resorting to any regularization.} 

\MMM To summarize, the novelty of our analysis is twofold. First, we prove the compactness of the compositions of magnetizations with deformations for sequences of admissible states with equi-bounded energies. This extends the compactness result obtained in \cite{kruzik.stefanelli.zeman} for incompressible materials to compressible ones. Moreover, this provides a more direct proof of the existence of minimizers for the functional in \eqref{eqn:energy-E-0} for compressible materials compared to  those available in \cite{barchiesi.desimone, barchiesi.henao.moracorral}. Note that, unlike in \cite{rybka.luskin}, no higher-order term is included in the magnetoelastic energy in the static setting. Second, in the quasistatic setting, we consider a more realistic notion of dissipation and we do not restrict ourselves to incompressible materials \cite{kruzik.stefanelli.zeman}. Solutions of the incremental minimization problem are shown to exist without resorting to any regularization. Finally, the existence of energetic solutions is achieved by including an additional energy term controlling the derivatives of the Jacobian cofator of deformations only, instead of the full Hessian matrix of deformations \cite{rybka.luskin}.
\EEE

{\MMM We remark that the choice to limit ourselves to the case of continuous deformations is taken just for convenience. We do not see substantial obstacles in extending our arguments to more general classes of possibly discontinuous deformations for which cavitation is excluded, like the ones considered in \cite{barchiesi.henao.moracorral,henao.stroffolini}, with the help of the techniques that have already been developed in these settings. Also, the global injectivity of admissible deformations is assumed in view of its physical interpretation, i.e. to avoid the interpenetration of matter, but this does not seem to be crucial for the analysis. It might be possible to achieve the same results without this assumption by relying on the local invertibility results available in the literature \cite{barchiesi.henao.moracorral,fonseca.gangbo,henao.stroffolini} in combination with suitable covering arguments.}

The paper is organized as follows. In Section \ref{sec:prel} we recall some preliminary results on the invertibility of Sobolev functions. Section \ref{sec:static} is devoted to {\MMM the analysis in the static setting including} the proof of Theorem \ref{thm:meta1}. Finally, Section \ref{sec:quasistatic} describes the quasistatic problem and contains the proof of Theorem \ref{thm:meta2}.

\section{Preliminaries}
\label{sec:prel}
In this section we collect some results regarding the invertibility of Sobolev maps with supercritical integrability. Let $\Omega\subset \R^3$ be a bounded Lipschitz domain. We consider maps in $ W^{1,p}(\Omega;\R^3)$ with $p>3$. Any such map admits a representative in $C^0(\closure{\Omega};\R^3)$ which has the {Lusin property $(N)$} \cite[Corollary 1]{marcus.mizel}, i.e. it maps sets of zero Lebesgue measure to sets of zero Lebesgue measure. Henceforth, \emph{we will always tacitly consider this representative}. In this case, the image of measurable sets is measurable and the area formula holds \cite[Corollary 2 and Theorem 2]{marcus.mizel}. As a consequence, if the Jacobian determinant is different from zero almost everywhere, then the map has also the {Lusin property $(N^{-1})$}, i.e. the preimage of every set with zero Lebesgue measure has zero Lebesgue measure. 

Let $\boldsymbol{y}\in W^{1,p}(\Omega;\R^3)$. To make up for the fact that $\boldsymbol{y}(\Omega)$ might not be open, even if $\det \nabla \boldsymbol{y}>0$ almost everywhere, we introduce the {deformed configuration}, which is defined as $\Omega^{\boldsymbol{y}}\coloneqq \boldsymbol{y}(\Omega)\setminus \boldsymbol{y}(\partial \Omega)$. To prove that this set is actually open, we employ the topological degree. Recall that the degree of $\boldsymbol{y}$ on $\Omega$ is a continuous map $\deg(\boldsymbol{y},\Omega,\cdot) \colon \R^3 \setminus \boldsymbol{y}(\partial \Omega)\to \Z$. For its definition and main properties, we refer to \cite{fonseca.gangbo2}.

\begin{lemma}[Deformed configuration]
\label{lem:deformed-configuration}
Let $\boldsymbol{y}\in W^{1,p}(\Omega;\R^3)$ be such that $\det \nabla \boldsymbol{y}>0$ almost everywhere in $\Omega$. Then, the deformed configuration $\Omega^{\boldsymbol{y}}$  is an open set that differs from $\boldsymbol{y}(\Omega)$ by at most a set of zero Lebesgue measure. Moreover $\closure{\Omega^{\boldsymbol{y}}}=\boldsymbol{y}(\closure{\Omega})$ and $\partial \Omega^{\boldsymbol{y}}=\boldsymbol{y}(\partial \Omega)$.
\end{lemma}
\begin{proof}
We claim that $\Omega^{\boldsymbol{y}}=\{\boldsymbol{\xi}\in \R^3 \setminus \boldsymbol{y}(\partial \Omega):\:\deg(\boldsymbol{y},\Omega,\boldsymbol{\xi})>0\}$. Once the claim is proved, we deduce that $\Omega^{\boldsymbol{y}}$ is open. Indeed, the set on the right-hand side is open  by the continuity of the degree. 

Let $\boldsymbol{\xi}_0\in \R^3 \setminus \boldsymbol{y}(\partial \Omega)$ be such that $\deg(\boldsymbol{y},\Omega,\boldsymbol{\xi}_0)>0$. Then, by the solvability property of the degree, $\boldsymbol{\xi}_0\in \boldsymbol{y}(\Omega)$ and, in turn, $\boldsymbol{\xi}_0\in \Omega^{\boldsymbol{y}}$. Conversely, let $\boldsymbol{\xi}_0\in \Omega^{\boldsymbol{y}}$. Denote by $V$ the connected component of $\R^3 \setminus \boldsymbol{y}(\partial \Omega)$ containing $\boldsymbol{\xi}_0$ and consider $R>0$ such that $B(\boldsymbol{\xi}_0,R)\subset \subset V$. Let $\psi \in C^\infty_c(\R^3)$ be such that $\psi\geq 0$, $\text{supp}\,\psi\, \subset\,\closure{B}(\boldsymbol{\xi}_0,R)\subset V$ and $\int_{\R^3} \psi\,\d \boldsymbol{\xi}=1$. Then, by the integral formula for the degree, we compute
\begin{equation*}
    \deg(\boldsymbol{y},\Omega,\boldsymbol{\xi})=\int_\Omega \psi \circ \boldsymbol{y}\,\det \nabla \boldsymbol{y}\,\d \boldsymbol{x}= \int_{\boldsymbol{y}^{-1}(B(\boldsymbol{\xi}_0,R))} \psi \circ \boldsymbol{y}\,\det \nabla \boldsymbol{y}\,\d \boldsymbol{x}.
\end{equation*}
As $\psi \circ \boldsymbol{y}>0$ on $\boldsymbol{y}^{-1}(B(\boldsymbol{\xi}_0,R))$ and $\det \nabla \boldsymbol{y}>0$ almost everywhere, we obtain $\deg(\boldsymbol{y},\Omega,\boldsymbol{\xi})>0$ and this proves the claim. 

By the Lusin property $(N)$, we have $\leb(\boldsymbol{y}(\Omega)\setminus \Omega^{\boldsymbol{y}})\leq \leb(\boldsymbol{y}(\partial \Omega))=0$. For simplicity, define $U\coloneqq \boldsymbol{y}^{-1}(\Omega^{\boldsymbol{y}})=\Omega \setminus \boldsymbol{y}^{-1}(\boldsymbol{y}(\partial \Omega))$. Then, $\Omega \setminus U=\boldsymbol{y}^{-1}(\boldsymbol{y}(\partial \Omega))$, so that $\leb(\Omega\setminus U)=0$ by the Lusin properties $(N)$ and $(N^{-1})$. In particular, $U$ is dense in $\Omega$.

We prove that $\closure{\Omega^{\boldsymbol{y}}}=\boldsymbol{y}(\closure{\Omega})$. As $\Omega^{\boldsymbol{y}}\subset \boldsymbol{y}(\Omega)$, we immediately have $\closure{\Omega^{\boldsymbol{y}}}\subset\closure{\boldsymbol{y}(\Omega)}=\boldsymbol{y}(\closure{\Omega})$. Let $\boldsymbol{\xi}\in \boldsymbol{y}(\closure{\Omega})$ and consider $\boldsymbol{x}\in \closure{\Omega}$ such that $\boldsymbol{y}(\boldsymbol{x})=\boldsymbol{\xi}$. By density, $\closure{U}=\closure{\Omega}$. Thus, there exists $(\boldsymbol{x}_n)\subset U$ such that $\boldsymbol{x}_n\to \boldsymbol{x}$ and, in turn, $\boldsymbol{\xi}_n\coloneqq\boldsymbol{y}(\boldsymbol{x}_n)\to \boldsymbol{\xi}$. As $(\boldsymbol{\xi}_n)\subset \Omega^{\boldsymbol{y}}$, this yields $\boldsymbol{\xi}\in \closure{\Omega^{\boldsymbol{y}}}$.

Finally, we prove that $\partial \Omega^{\boldsymbol{y}}=\boldsymbol{y}(\partial \Omega)$. This follows combining
\begin{equation*}
        \partial \Omega^{\boldsymbol{y}}=\closure{\Omega^{\boldsymbol{y}}}\setminus(\Omega^{\boldsymbol{y}})^\circ=\boldsymbol{y}(\closure{\Omega})\setminus \Omega^{\boldsymbol{y}}=(\boldsymbol{y}(\closure{\Omega})\setminus \boldsymbol{y}(\Omega)) \cup (\boldsymbol{y}(\closure{\Omega})\cap \boldsymbol{y}(\partial \Omega))\subset \boldsymbol{y}(\partial \Omega)  
\end{equation*}
and
\begin{equation*}
        \partial \Omega^{\boldsymbol{y}}=\closure{\Omega^{\boldsymbol{y}}}\cap \closure{\R^3 \setminus \Omega^{\boldsymbol{y}}}=\boldsymbol{y}(\closure{\Omega}) \cap \closure{(\R^3 \setminus \boldsymbol{y}(\Omega))\cup \boldsymbol{y}(\partial \Omega)} \supset \boldsymbol{y}(\closure{\Omega}) \cap \boldsymbol{y}(\partial \Omega)=\boldsymbol{y}(\partial \Omega).
\end{equation*}
\end{proof}

 The next example clarifies the difference between the sets $\boldsymbol{y}(\Omega)$ and $\Omega^{\boldsymbol{y}}$. 

\begin{example}[Ball's example]
\label{ex:ball}
The following is inspired by \cite[Example 1]{ball}. Let $\Omega\coloneqq(-1,1)^3$ and write $\Omega=\Omega^+ \cup P \cup \Omega^-$, where
\begin{equation*}
    \Omega^+\coloneqq (0,1)\times (-1,1)^2, \quad  P\coloneqq \{0\}\times (-1,1)^2, \quad \Omega^-\coloneqq (-1,0)\times (-1,1)^2.
\end{equation*}
Define  $\boldsymbol{y}\colon \Omega \to \R^3$ by  $\boldsymbol{y}(\boldsymbol{x})\coloneqq(x_1,x_2,|x_1|\,x_3)$, where $\boldsymbol{x}=(x_1,x_2,x_3)$. The corresponding deformed set is depicted in Figure \ref{fig:ball-ex}.
Then $\boldsymbol{y}\in W^{1,\infty}(\Omega;\R^3)$ and for every $\boldsymbol{x}\in \Omega \setminus P$ we have
\begin{equation*}
    \nabla \boldsymbol{y}(\boldsymbol{x})=
    \begin{pmatrix}
      1  & 0 & 0 \\
      0 & 1 & 0\\
      {x_1\,x_3}/{|x_1|} & 0 & |x_1|
    \end{pmatrix}.
\end{equation*}
In particular, $\det \nabla \boldsymbol{y}>0$ on $\Omega \setminus P$. We have $\boldsymbol{y}(\Omega^+)=V^+$, $\boldsymbol{y}(P)=S$ and $\boldsymbol{y}(\Omega^-)=V^-$, where, for $\boldsymbol{\xi}=(\xi_1,\xi_2,\xi_3)$, we set
\begin{equation*}
\begin{split}
    V^+&\coloneqq \{\boldsymbol{\xi}\in \R^3: 0<\xi_1<1,\:-1<\xi_2<1,\:|\xi_3|<\xi_1\},\\
    S&\coloneqq \{0\}\times (-1,1)\times \{0\},\\
    V^-&\coloneqq \{\boldsymbol{\xi}\in \R^3: -1<\xi_1<0,\:-1<\xi_2<1,\:|\xi_3|<-\xi_1\}.
\end{split}
\end{equation*}
Note that $\boldsymbol{y}\restr{\Omega \setminus P}$ is injective, but $\boldsymbol{y}$ is not a homeomorphism. Also, $\boldsymbol{y}(\Omega)=V^+ \cup S \cup V^-$ is not open. Instead, $\Omega^{\boldsymbol{y}}=V^+ \cup V^-$, since $S\subset\boldsymbol{y}\left (\closure{P}\cap \partial \Omega \right)$, and this set is open. Note also that, while $\boldsymbol{y}(\Omega)$ is necessarily connected, the deformed configuration $\Omega^{\boldsymbol{y}}$ is not.

\begin{figure}
    \begin{tikzpicture}
    \node[anchor=south west,inner sep=0] at (0,0)
    {\includegraphics[width=\textwidth]{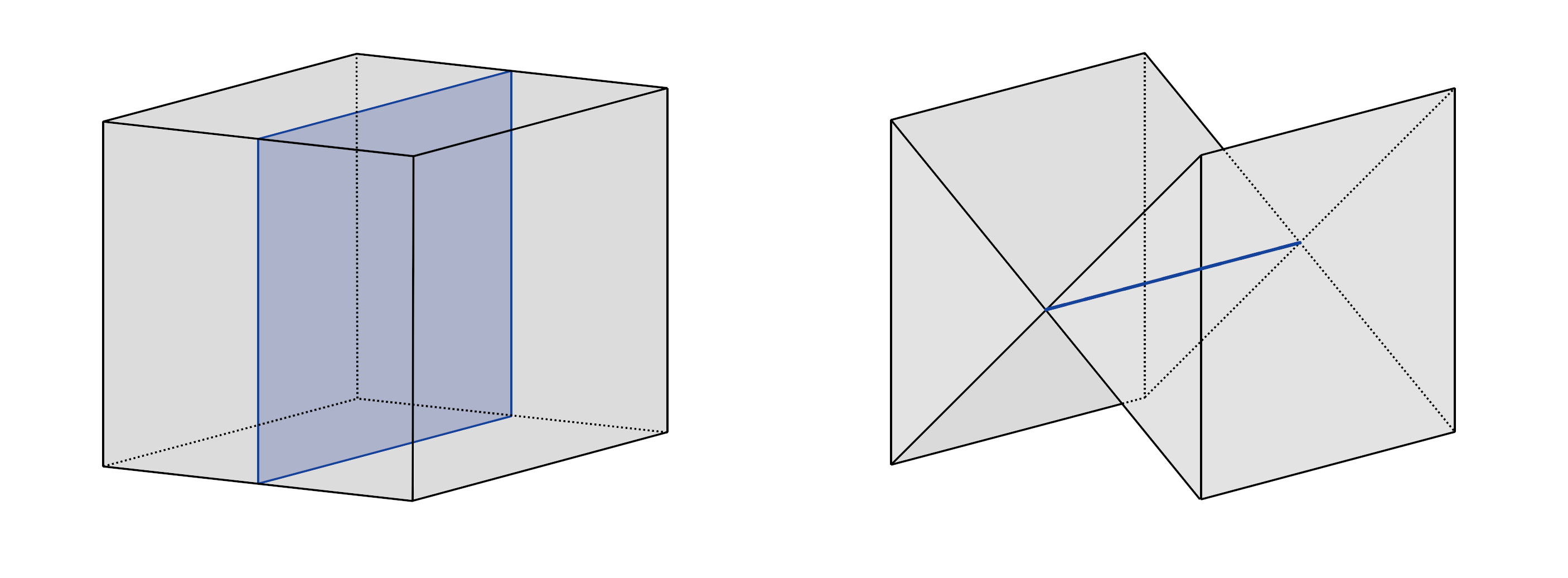}};
    \node[color=mblue] at (3,3) {$P$}; 
    \node at (2,4) {$\Omega^-$};
    \node at (6,4) {$\Omega^+$};
    \node at (10.2,4.3) {$V^-$};
    \node at (13.8,4.3) {$V^+$};
    \node[color=mblue] at (12.5,3.5) {$S$};
    \end{tikzpicture}
    \caption{The deformation in Example \ref{ex:ball}.}
    \label{fig:ball-ex}
\end{figure}
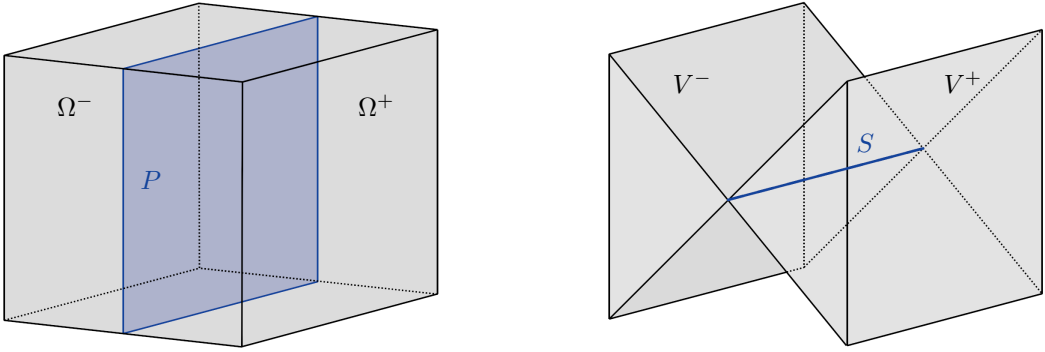

\end{example}

\begin{remark}[Topological image]
\label{rem:topological-image}
Let $\boldsymbol{y}\in W^{1,p}(\Omega;\R^3)$. The {topological image} of $\boldsymbol{y}$ is given by the set $\text{im}_{\rm T}(\boldsymbol{y},\Omega)\coloneqq \{\boldsymbol{\xi}\in \R^3 \setminus \boldsymbol{y}(\partial \Omega):\:\deg(\boldsymbol{y},\Omega,\boldsymbol{\xi})\neq 0\}$. Note that $\deg(\boldsymbol{y},\Omega,\boldsymbol{\xi})=0$ for every $\boldsymbol{\xi}\in \R^3 \setminus \boldsymbol{y}(\closure{\Omega})$, so that $\text{im}_{\rm T}(\boldsymbol{y},\Omega)\subset \boldsymbol{y}(\Omega)$. In  relation with the problem of invertibility of deformations in elasticity, the topological image was first considered in \cite{sverak} and then in several other contributions \cite{barchiesi.henao.moracorral,bouchala.hencl.molchanova,henao.moracorral,henao.moracorral2,henao.moracorral.oliva,mueller.qi.yan,mueller.spector,qi}. {\MMM In} Lemma \ref{lem:deformed-configuration}, we proved that, if $\det \nabla \boldsymbol{y}>0$ almost everywhere, then
\begin{equation*}
    \Omega^{\boldsymbol{y}}=\text{im}_{\rm T}(\boldsymbol{y},\Omega)=\{\boldsymbol{\xi}\in \R^3 \setminus \boldsymbol{y}(\partial \Omega):\:\deg(\boldsymbol{y},\Omega,\boldsymbol{\xi})> 0\}.
\end{equation*}
For more information about the topological properties of Sobolev maps with supercritical integrability, we refer to \cite{kromer}. 
\end{remark}

We now consider the invertibility of Sobolev maps with supercritical integrability. 
Let $\boldsymbol{y}\in W^{1,p}(\Omega;\R^3)$  with {\MMM $p>3$ be such that} $\det \nabla \boldsymbol{y}>0$ almost everywhere. Assume that $\boldsymbol{y}$ is {almost everywhere injective}, i.e.  there exists a  set $X \subset \Omega$ with $\leb(X)=0$ such that $\boldsymbol{y}\restr{\Omega \setminus X}$ is injective. In this case, we can consider the inverse  $\boldsymbol{y}\restr{\Omega\setminus X}^{-1}\colon \boldsymbol{y}(\Omega\setminus X)\to \Omega \setminus X$. Note that $\leb(\boldsymbol{y}(X))=0$ by the Lusin property $(N)$. We define the map $\boldsymbol{v} \colon \Omega^{\boldsymbol{y}} \to \R^3$ by setting
\begin{equation}
\label{eqn:inverse-deformation}
    \boldsymbol{v}(\boldsymbol{\xi})\coloneqq \begin{cases} \boldsymbol{y}\restr{\Omega\setminus X}^{-1}(\boldsymbol{\xi}), & \text{if $\boldsymbol{\xi}\in \Omega^{\boldsymbol{y}}\setminus \boldsymbol{y}(X)$,}\\ \boldsymbol{a}, & \text{if $\boldsymbol{\xi}\in \Omega^{\boldsymbol{y}}\cap \boldsymbol{y}(X)$,}
    \end{cases}
\end{equation}
where $\boldsymbol{a}\in \R^3$ is arbitrarily fixed. The map $\boldsymbol{v}$ satisfies $\boldsymbol{v}\circ \boldsymbol{y}=\boldsymbol{id}$ almost everywhere in $\Omega$ and $\boldsymbol{y}\circ \boldsymbol{v}=\boldsymbol{id}$ almost everywhere in $\Omega^{\boldsymbol{y}}$. Since $\boldsymbol{y}$  maps measurable sets to measurable sets, the measurability of $\boldsymbol{v}$ follows. As $\boldsymbol{y}$ has both Lusin properties $(N)$ and $(N^{-1})$, the map $\boldsymbol{v}$ has the same properties. Moreover, $\boldsymbol{v}\in L^\infty(\Omega^{\boldsymbol{y}};\R^3)$ since $\boldsymbol{v}(\Omega^{\boldsymbol{y}})\subset \Omega \cup \{\boldsymbol{a}\}$ and $\Omega$ is bounded.  

We remark that the definition of $\boldsymbol{v}$ in \eqref{eqn:inverse-deformation}  depends on the choice of the set $X$ where $\boldsymbol{y}$ is not injective and of the value $\boldsymbol{a}\in \R^3$. However, as $\boldsymbol{y}$ has the Lusin property $(N)$, its equivalence class is uniquely determined and coincides with the one of the classical inverse $\boldsymbol{y}^{-1}$, where the latter is defined out of a subset of $\boldsymbol{y}(\Omega)$ with zero Lebesgue measure. Hence, with a slight abuse of notation, we will denote this equivalence class of functions defined on $\Omega^{\boldsymbol{y}}$ by $\boldsymbol{y}^{-1}$ and we will refer to it as \emph{the} inverse of $\boldsymbol{y}$. 

\begin{remark}[Ciarlet-Ne\v{c}as condition]
\label{rem:ciaret-necas}
Let $\boldsymbol{y}\in W^{1,p}(\Omega;\R^3)$ be such that $\det \nabla \boldsymbol{y}>0$ almost everywhere. Then, $\boldsymbol{y}$ is almost everywhere injective if and only if it satisfies the {Ciarlet-Ne\v{c}as condition} \cite{ciarlet.necas}, which reads
\begin{equation*}
    \int_\Omega \det \nabla \boldsymbol{y}\,\d\boldsymbol{x}\leq \leb(\boldsymbol{y}(\Omega)).
\end{equation*}
This equivalence easily follows from the area formula \cite[p. 185]{ciarlet.necas}.
Note that the Ciarlet-Ne\v{c}as condition is preserved under weak convergence in $W^{1,p}(\Omega;\R^3)$ thanks to the weak continuity of {\MMM Jacobian} minors and the Morrey embedding.
As a consequence, given $(\boldsymbol{y}_n)\subset W^{1,p}(\Omega;\R^3)$ such that each $\boldsymbol{y}_n$ is almost everywhere injective with $\det \nabla \boldsymbol{y}_n>0$ almost everywhere, if $\boldsymbol{y}_n\wk\boldsymbol{y}$ in $W^{1,p}(\Omega;\R^3)$ for some $\boldsymbol{y}\in W^{1,p}(\Omega;\R^3)$ with $\det \nabla \boldsymbol{y}>0$ almost everywhere, then $\boldsymbol{y}$ is almost everywhere injective. Note that the condition $\det \nabla \boldsymbol{y}>0$ almost everywhere has to be assumed a priori.
\end{remark}

The inverse $\boldsymbol{y}^{-1}$ of $\boldsymbol{y}$ turns out to have Sobolev regularity. Note that this makes sense since, by definition, $\boldsymbol{y}^{-1}$ is defined on the deformed configuration $\Omega^{\boldsymbol{y}}$, which is open by Lemma \ref{lem:deformed-configuration}.
The Sobolev regularity of the inverse has been proved for more general classes of deformations, such as in \cite[Proposition 5.3]{barchiesi.henao.moracorral}, \cite[Theorem 9.3]{henao.moracorral.oliva},  \cite[Theorem 4.6]{qi},  and \cite[Theorem 8]{sverak}. For convenience of the reader, we recall the proof. Note that, here, {\MMM we do not impose any regularity on the boundary}. 

\begin{proposition}[Global invertibility]
\label{prop:global-invertibility}
Let $\boldsymbol{y}\in W^{1,p}(\Omega;\R^3)$ be almost everywhere injective with  $\det \nabla \boldsymbol{y}>0$ almost everywhere.  Then, $\boldsymbol{y}^{-1}\in W^{1,1}(\Omega^{\boldsymbol{y}};\R^3)$ with $\nabla \boldsymbol{y}^{-1}=(\nabla \boldsymbol{y})^{-1}\circ \boldsymbol{y}^{-1}$ almost everywhere in $\Omega^{\boldsymbol{y}}$. Moreover,   $\cof\,\nabla \boldsymbol{y}^{-1}\in L^1(\Omega^{\boldsymbol{y}};\rtt)$ and $\det \nabla \boldsymbol{y}^{-1}\in L^1(\Omega^{\boldsymbol{y}})$. 
\end{proposition}

\begin{proof}
By the Piola identity, we have
\begin{equation}
    \label{eqn:piola}
    \int_\Omega \cof\,\nabla \boldsymbol{y}:\nabla \boldsymbol{\zeta}\,\d\boldsymbol{x}=0
\end{equation}
for every $\boldsymbol{\zeta}\in C^\infty_c(\Omega;\R^3)$. {\MMM As $\cof \,\nabla \boldsymbol{y}\in L^{{p}/{2}}(\Omega;\rtt)$ and $p/2>p/(p-1)=p'$ since $p>3$,} by density, this actually holds for $\boldsymbol{\zeta}\in W^{1,p}_0(\Omega;\R^3)$. Let $\varphi \in C^\infty(\closure{\Omega})$ and $\boldsymbol{\psi}\in C^\infty_c(\Omega^{\boldsymbol{y}};\R^3)$. Choosing $\boldsymbol{\zeta} = \varphi\,\boldsymbol{\psi}\circ \boldsymbol{y}$ in \eqref{eqn:piola}, after some algebraic manipulations, we obtain the following identity:
\begin{equation}
    \label{eqn:div-identity}
    -\int_\Omega \varphi\,\div \boldsymbol{\psi}\circ \boldsymbol{y}\,\det \nabla \boldsymbol{y}\,\d\boldsymbol{x}=\int_\Omega \boldsymbol{\psi}\circ \boldsymbol{y} \otimes \nabla \varphi : {\rm cof}\,\nabla \boldsymbol{y}\,\d\boldsymbol{x}.
\end{equation}

Let $X \subset \Omega$ with $\leb(X)=0$ be such that $\boldsymbol{y}\restr{\Omega\setminus X}$ is injective. For clarity, let us consider the representative $\boldsymbol{v}$ of $\boldsymbol{y}^{-1}$ in \eqref{eqn:inverse-deformation} and let us fix a representative of $\nabla \boldsymbol{y}$. Set 
$D \coloneqq \Omega \setminus (\boldsymbol{y}^{-1}(\boldsymbol{y}(\partial \Omega))\cup \{\det \nabla \boldsymbol{y}\leq 0\} \cup X)$, so that $\boldsymbol{v}=\boldsymbol{y}\restr{D}^{-1}$ on $\boldsymbol{y}(D)$ and $\nabla \boldsymbol{y}$ is invertible on $D$.
Let $\boldsymbol{\Phi}\in C^\infty_c(\Omega^{\boldsymbol{y}};\rtt)$ and denote its rows by $\boldsymbol{\Phi}^i=(\Phi^i_1,\Phi^i_2,\Phi^i_3)^\top$, where $i=1,2,3$. Using the change-of-variable formula, we compute
\begin{equation*}
    \begin{split}
        - \int_{\Omega^{\boldsymbol{y}}} \boldsymbol{v}\cdot \div \boldsymbol{\Phi}\,\d\boldsymbol{\xi}&=-\int_{\boldsymbol{y}(D)} \boldsymbol{y}\restr{D}^{-1}\cdot \div \boldsymbol{\Phi} \,\d\boldsymbol{\xi}=-\int_D \boldsymbol{x}\cdot \div \boldsymbol{\Phi} \circ \boldsymbol{y}\,\det \nabla \boldsymbol{y}\,\d\boldsymbol{x}\\
        &=-\int_\Omega \boldsymbol{x}\cdot \div \boldsymbol{\Phi} \circ \boldsymbol{y}\,\det \nabla \boldsymbol{y}\,\d\boldsymbol{x}=-\sum_{i=1}^3 \int_\Omega x_i\,\div \boldsymbol{\Phi}^i \circ \boldsymbol{y}\,\det \nabla \boldsymbol{y}\,\d\boldsymbol{x}.
    \end{split}
\end{equation*}
Then using \eqref{eqn:div-identity} with $\varphi(\boldsymbol{x})=x_i$ for every $\boldsymbol{x}\in \Omega$ and $\boldsymbol{\psi}(\boldsymbol{\xi})=\boldsymbol{\Phi}^i(\boldsymbol{\xi})$ for every $\boldsymbol{\xi}\in \Omega^{\boldsymbol{y}}$, we obtain 
\begin{equation*}
    \begin{split}
        - \int_{\Omega^{\boldsymbol{y}}} \boldsymbol{v}\cdot \div \boldsymbol{\Phi}\,\d\boldsymbol{\xi}
        &=\sum_{i,j=1}^3 \int_\Omega \Phi^i_j \circ \boldsymbol{y}\,(\cof\,\nabla \boldsymbol{y})^j_i\,\d\boldsymbol{x}=\sum_{i,j=1}^3 \int_\Omega \Phi^i_j \circ \boldsymbol{y}\,(\adj\,\nabla \boldsymbol{y})^i_j\,\d\boldsymbol{x}\\
        &=\int_\Omega \boldsymbol{\Phi}\circ \boldsymbol{y}:\adj\,\nabla \boldsymbol{y}\,\d\boldsymbol{x}=\int_\Omega \boldsymbol{\Phi}\circ \boldsymbol{y}:(\nabla \boldsymbol{y})^{-1}\,\det \nabla \boldsymbol{y}\,\d\boldsymbol{x}\\
        &=\int_D \boldsymbol{\Phi}\circ \boldsymbol{y}:(\nabla \boldsymbol{y})^{-1}\,\det \nabla \boldsymbol{y}\,\d\boldsymbol{x}  =\int_{\boldsymbol{y}(D)} \boldsymbol{\Phi}:(\nabla \boldsymbol{y})^{-1}\circ \boldsymbol{y}\restr{D}^{-1}\,\d \boldsymbol{\xi},
    \end{split}
\end{equation*}
where, in the last line, we used again the change-of-variable formula. Hence, as $\leb(\Omega^{\boldsymbol{y}}\setminus \boldsymbol{y}(D))=0$, we deduce that $\boldsymbol{v}$ admits a weak gradient with a representative given by  
\begin{equation*}
    \nabla\boldsymbol{v}(\boldsymbol{\xi})\coloneqq 
    \begin{cases} 
    (\nabla \boldsymbol{y})^{-1}\circ \boldsymbol{y}\restr{D}^{-1}(\boldsymbol{\xi}) & \text{if $\boldsymbol{\xi}\in \boldsymbol{y}(D)$,}\\
    \boldsymbol{A} & \text{if $\boldsymbol{\xi}\in \Omega^{\boldsymbol{y}}\setminus \boldsymbol{y}(D)$,}
    \end{cases}
\end{equation*}
where $\boldsymbol{A}\in \rtt$ is arbitrary.  Thanks to the Lusin property $(N)$, the equivalence class of $\nabla \boldsymbol{v}$ is uniquely determined. Moreover, it belongs to $L^1(\Omega^{\boldsymbol{y}};\rtt)$. Indeed, by the change-of-variable formula
\begin{equation*}
    \begin{split}
        \int_{\Omega^{\boldsymbol{y}}} |\nabla \boldsymbol{v}|\,\d\boldsymbol{\xi}&=\int_{\boldsymbol{y}(D)} |(\nabla \boldsymbol{y})^{-1}|\circ \boldsymbol{y}\restr{D}^{-1}\,\d\boldsymbol{\xi}   =\int_D |(\nabla \boldsymbol{y})^{-1}|\det \nabla \boldsymbol{y}\,\d\boldsymbol{x}\\
        &=\int_D |\adj\,\nabla\boldsymbol{y}|\,\d\boldsymbol{x}
        =\int_\Omega |\adj\,\nabla\boldsymbol{y}|\,\d\boldsymbol{x}.
    \end{split}
\end{equation*}
Thus, $\boldsymbol{v}\in W^{1,1}(\Omega^{\boldsymbol{y}};\R^3)$. Similarly, using the identity $\adj\,(\boldsymbol{F}^{-1})=(\det \boldsymbol{F})^{-1}\boldsymbol{F}$ for every $\boldsymbol{F}\in \rtt_+$, we compute
\begin{equation*}
    \begin{split}
        \int_{\Omega^{\boldsymbol{y}}}|\adj\,\nabla \boldsymbol{v}|\,\d\boldsymbol{\xi}&=\int_{\boldsymbol{y}(D)} (\det \nabla \boldsymbol{y})^{-1}\circ \boldsymbol{y} \restr{D}^{-1} \,|\nabla \boldsymbol{y}|\circ \boldsymbol{y}\restr{D}^{-1} \,\d \boldsymbol{\xi}=\int_D |\nabla \boldsymbol{y}|\,\d\boldsymbol{x}=\int_{\Omega} |\nabla \boldsymbol{y}|\,\d\boldsymbol{x},
    \end{split}
\end{equation*}
while, using the identity $\det (\boldsymbol{F}^{-1})=(\det \boldsymbol{F})^{-1}$ in $\boldsymbol{F}\in \rtt_+$, we obtain
\begin{equation*}
    \int_{\Omega^{\boldsymbol{y}}} \det \nabla \boldsymbol{v}\,\d\boldsymbol{x}=\int_{\boldsymbol{y}(D)} (\det \nabla \boldsymbol{y})^{-1}\circ \boldsymbol{y}\restr{D}^{-1}\,\d\boldsymbol{\xi}=\leb(D)=\leb(\Omega).
\end{equation*}
Therefore, $\cof\,\nabla \boldsymbol{v}\in L^1(\Omega^{\boldsymbol{y}};\rtt)$ and $\det \nabla \boldsymbol{v}\in L^1(\Omega^{\boldsymbol{y}})$.
\end{proof}

\begin{remark}[Area formula for the inverse]
\label{rem:area-formula-inverse}
Let $\boldsymbol{y}\in W^{1,p}(\Omega;\R^3)$ be almost everywhere injective with $\det \nabla \boldsymbol{y}>0$ almost everywhere. Let $X \subset \Omega$ with $\leb(X)=0$ be such that $\boldsymbol{y}\restr{\Omega \setminus X}$ is injective and let
$\boldsymbol{v}$ be the representative of $\boldsymbol{y}^{-1}$ in \eqref{eqn:inverse-deformation}. By Proposition \ref{prop:global-invertibility},  $\boldsymbol{v}\in W^{1,1}(\Omega^{\boldsymbol{y}};\R^3)$. Since $\boldsymbol{y}$ has the Lusin property $(N^{-1})$, the map $\boldsymbol{v}$ has the Lusin property $(N)$. Moreover, $\boldsymbol{v}$ is almost everywhere injective. Thus, we can use the area formula to estimate the measure of preimages of sets via $\boldsymbol{y}$ {\MMM as follows}. {\MMM For every $F\subset \R^3$, we write $\boldsymbol{y}^{-1}(F)\coloneqq\{\boldsymbol{x}\in \Omega:\:\boldsymbol{y}(\boldsymbol{x})\in F\}$.
Let {\MMM $F$} be measurable. By \eqref{eqn:inverse-deformation}, we have $\boldsymbol{v}(F)=\boldsymbol{y}^{-1}(F\setminus(\boldsymbol{y}(\partial \Omega)\cup\boldsymbol{y}(X)))$. Then, exploiting both Lusin properties $(N)$ and $(N^{-1})$ of $\boldsymbol{y}$ and} applying the area formula \cite[Theorem 2]{hajlasz}, we compute
\begin{equation*}
    \leb(\boldsymbol{y}^{-1}(F))=\leb(\boldsymbol{v}(F))=\int_F \det \nabla \boldsymbol{v}\,\d\boldsymbol{\xi}.
\end{equation*}
\end{remark}

\section{Static Setting}
\label{sec:static}
\MMM
\subsection{The mathematical model}
\label{subs:mathematical-model}
\EEE
Let $\Omega \subset \R^3$ be a bounded Lipschitz domain. For $p>3$ fixed, the class of admissible deformations is given by
\begin{equation}
    \label{eqn:admissible-deformations}
    \mathcal{Y}\coloneqq \left \{ \boldsymbol{y}\in W^{1,p}(\Omega;\R^3):\:\det \nabla \boldsymbol{y}>0\:\text{a.e.},\:\text{$\boldsymbol{y}$ a.e. injective},\:\boldsymbol{y}=\overline{\boldsymbol{y}}\:\:{ \text{ on}}\:\Gamma  \right \},
\end{equation}
where $\Gamma \subset \partial \Omega$ {\MMM is $\haus$-measurable} with $\haus(\Gamma)>0$ and $\overline{\boldsymbol{y}}\in C^0(\Gamma;\R^3)$ {\MMM \st{are given}}.

\begin{example}
\label{ex:ball2}
Let $\Omega$  and $\boldsymbol{y}$ be as in Example \ref{ex:ball}. Given  $\Gamma\coloneqq \{-1,1\}\times (-1,1)^2$ and $\overline{\boldsymbol{y}}\coloneqq \boldsymbol{id}$, we have  $\boldsymbol{y}\in \mathcal{Y}$. In particular, this is a case in which $\mathcal{Y}\neq \emptyset$. 
\end{example}

Henceforth, we identify each $\boldsymbol{y}\in \mathcal{Y}$  with its continuous representative and we set $\Omega^{\boldsymbol{y}}\coloneqq \boldsymbol{y}(\Omega)\setminus \boldsymbol{y}(\partial \Omega)$. Then, admissible magnetizations are given by maps $\boldsymbol{m}\in W^{1,2}(\Omega^{\boldsymbol{y}};\S^2)$. Note that this makes sense as $\Omega^{\boldsymbol{y}}$ is open by Lemma \ref{lem:deformed-configuration}. Thus, the class of admissible states is defined as
\begin{equation}
    \label{eqn:admissible-states}
    \mathcal{Q}\coloneqq \left \{(\boldsymbol{y},\boldsymbol{m})\, :\boldsymbol{y}\in \mathcal{Y},\: \boldsymbol{m}\in W^{1,2}(\Omega^{\boldsymbol{y}};\S^2) \right \}.
\end{equation}
We endow the set $\mathcal{Q}$ with the topology that makes the map $\boldsymbol{q}=(\boldsymbol{y},\boldsymbol{m})\mapsto (\boldsymbol{y},\chi_{\Omega^{\boldsymbol{y}}}\boldsymbol{m},\chi_{\Omega^{\boldsymbol{y}}}\nabla\boldsymbol{m})$ from $\mathcal{Q}$ to $W^{1,p}(\Omega;\R^3)\times L^2(\R^3;\R^3)\times L^2(\R^3;\rtt)$ a homeomorphism onto its image, where the latter space is equipped with the weak product topology. Hence $\boldsymbol{q}_n\to \boldsymbol{q}$ in $\mathcal{Q}$ if and only if the following convergences hold:
\begin{equation}
    \label{eqn:Q-convergence-y}
    \text{$\boldsymbol{y}_n \wk \boldsymbol{y}$ in $W^{1,p}(\Omega;\R^3)$,}
\end{equation}
\begin{equation}
    \label{eqn:Q-convergence-m}
    \text{$\chi_{\Omega^{\boldsymbol{y}_n}}\boldsymbol{m}_n \wk \chi_{\Omega^{\boldsymbol{y}}}\boldsymbol{m}$ in $L^2(\R^3;\R^3)$,}
\end{equation}
\begin{equation}
    \label{eqn:Q-convergence-nabla-m}
    \text{$\chi_{\Omega^{\boldsymbol{y}_n}}\nabla\boldsymbol{m}_n \wk \chi_{\Omega^{\boldsymbol{y}}}\nabla\boldsymbol{m}$ in $L^2(\R^3;\rtt)$.}
\end{equation}
In this case, up to subsequences, we actually have
\begin{equation}
    \label{eqn:Q-convergence-strong}
    \text{$\chi_{\Omega^{\boldsymbol{y}_n}}\boldsymbol{m}_n \to \chi_{\Omega^{\boldsymbol{y}}}\boldsymbol{m}$ in $L^a(\R^3;\R^3)$ for every $1 \leq a < \infty$.}
\end{equation} 

The energy functional $E \colon \mathcal{Q}\to \R$ is defined, for $\boldsymbol{q}=(\boldsymbol{y},\boldsymbol{m})\in\mathcal{Q}$, by setting
\begin{equation}
    \label{eqn:energy-E}
    E(\boldsymbol{q})\coloneqq \int_\Omega W(\nabla\boldsymbol{y},\boldsymbol{m}\circ \boldsymbol{y})\,\d\boldsymbol{x}+ \alpha \int_{\Omega^{\boldsymbol{y}}}|\nabla \boldsymbol{m}|^2\,\d \boldsymbol{\xi}+\frac{\mu_0}{2}\int_{\R^3}|\nabla \zeta_{\boldsymbol{m}}|^2\,\d\boldsymbol{\xi}. 
\end{equation}
The first term represents the {elastic energy} of the system. Note that, as $\boldsymbol{y}$ satisfies the Lusin property $(N^{-1})$, the composition $\boldsymbol{m}\circ \boldsymbol{y}$ is measurable and its equivalence class does not depend on the choice of the representative of $\boldsymbol{m}$. The nonlinear elastic energy density $W \colon \rtt_+ \times \S^2 \to [0,+\infty)$ is continuous and satisfies the following two assumptions:
\begin{align}
    \nonumber \textbf{(coercivity)}& \hspace{3mm}\text{there exist a constant $K>0$ and a Borel function $\gamma \colon [0,+\infty) \to [0,+\infty)$ satisfying}\\
    \nonumber & \hspace{3mm}\text{$\lim_{h \to 0^+}\gamma(h)= +\infty$ such that}\\
    & \hspace{32mm}{\displaystyle W(\boldsymbol{F},\boldsymbol{\lambda})\geq K |\boldsymbol{F}|^p+\gamma(\det \boldsymbol{F})} \label{eqn:W-coercivity}\\
    \nonumber & \hspace{3mm}\text{for every $\boldsymbol{F}\in \rtt_+$ and $\boldsymbol{\lambda}\in \S^2$};\\
    \nonumber \textbf{(polyconvexity)}& \hspace{3mm}\text{there exists a function $\widehat{W}\colon \rtt_+ \times \rtt_+ \times \R_+ \times \S^2 \to [0,+\infty)$ such that $\widehat{W}(\cdot,\cdot,\cdot,\boldsymbol{\lambda})$}\\
    \nonumber & \hspace{3mm}\text{is convex for every $\boldsymbol{\lambda}\in \S^2$ and there holds}\\
    & \hspace{30mm}\text{$W(\boldsymbol{F},\boldsymbol{\lambda})=\widehat{W}(\boldsymbol{F},\cof\,\boldsymbol{F},\det \boldsymbol{F},\boldsymbol{\lambda})$}\label{eqn:W-polyconvex}\\
    \nonumber & \hspace{3mm}\text{for every $\boldsymbol{F}\in \rtt_+$ and $\boldsymbol{\lambda}\in \S^2$.}
\end{align}
{\MMM Another standard assumption on the elastic energy density is the one of {frame-indifference}, which reads
\begin{equation*}
    \forall\,\boldsymbol{F}\in\rtt_+,\:\forall\,\boldsymbol{\lambda}\in\S^2,\:\forall\,\boldsymbol{Q}\in SO(3),\quad W(\boldsymbol{Q}\boldsymbol{F},\boldsymbol{Q}\boldsymbol{\lambda})=W(\boldsymbol{F},\boldsymbol{\lambda}).
\end{equation*}
This assumption is crucial from the physical point of view, as it ensures the objectivity of the magnetoelastic energy. However, this requirement will play no role in our analysis.}

The second term in \eqref{eqn:energy-E} is the {exchange energy} and comprises the parameter $\alpha>0$. The third term is called {magnetostatic energy} and involves the function $\zeta_{\boldsymbol{m}}\colon \R^3 \to \R$ which is a weak solution of the magnetostatic Maxwell equation:
\begin{equation}
    \Delta \zeta_{\boldsymbol{m}}=\div(\chi_{\Omega^{\boldsymbol{y}}}\boldsymbol{m})\:\text{in}\:\R^3.
\end{equation}
This means that $\zeta_{\boldsymbol{m}}$ belongs to the homogeneous Sobolev space
\begin{equation*}
    V^{1,2}(\R^3)\coloneqq \{\varphi \in L^2_\loc(\R^3):\:\nabla \varphi \in L^2(\R^3;\R^3)\}
\end{equation*}
and satisfies the following:
\begin{equation*}
    \forall\, \varphi \in V^{1,2}(\R^3),\quad \int_{\R^3} \nabla \zeta_{\boldsymbol{m}}\cdot \nabla \varphi\,\d\boldsymbol{\xi}=\int_{\R^3} \chi_{\Omega^{\boldsymbol{y}}}\boldsymbol{m} \cdot \nabla \varphi\,\d\boldsymbol{\xi}.
\end{equation*}
Note that such weak solutions exist and are unique up to additive constants \cite[Proposition 8.8]{barchiesi.henao.moracorral}, so that their gradient is uniquely defined. 

\subsection{Compactness and existence of minimizers}
\label{subs:compactness-existence-minimizers}
The main result of this section is the existence of minimizers of the energy $E$ in \eqref{eqn:energy-E}. Recall the definition of the class of admissible states in \eqref{eqn:admissible-deformations}--\eqref{eqn:admissible-states}.

\begin{theorem}[Existence of minimizers]
\label{thm:existence-minimizers}
Assume $p>3$ and $\mathcal{Y}\neq \emptyset$. Suppose that $W$ is continuous and satisfies \eqref{eqn:W-coercivity}--\eqref{eqn:W-polyconvex}. Then, the functional $E$ admits a minimizer in $\mathcal{Q}$.
\end{theorem}

\begin{remark}[Applied loads]
{\MMM In Theorem \ref{thm:existence-minimizers}, applied loads can be also taken into account.} Let $\boldsymbol{f}\in L^{p'}(\Omega;\R^3)$, $\boldsymbol{g}\in L^{p'}(\Sigma;\R^3)$ where $\Sigma\subset \partial \Omega$ is $\haus$-measurable and such that $\Gamma \cap \Sigma=\emptyset$, and $\boldsymbol{h}\in L^2(\R^3;\R^3)$, represent an applied body force, surface force and magnetic field, respectively. {\MMM Then, the work of applied loads, which should be subtracted from the magnetolastic energy, is described by the functional $L\colon \mathcal{Q}\to\R$} given by
\begin{equation*}
    L(\boldsymbol{q})\coloneqq\int_\Omega \boldsymbol{f}\cdot \boldsymbol{y}\,\d \boldsymbol{x}+\int_\Sigma \boldsymbol{g}\cdot \boldsymbol{y}\,\d \haus+\int_{\Omega^{\boldsymbol{y}}} \boldsymbol{h}\cdot \boldsymbol{m}\,\d\boldsymbol{\xi},
\end{equation*}
where $\boldsymbol{q}=(\boldsymbol{y},\boldsymbol{m})$. {\MMM Note that the energy contribution determined by the external magnetic field, usually called {Zeemann energy}, is described by an Eulerian term. The functional $L$ is evidently continuous with respect to the topology of $\mathcal{Q}$, so that its treatment is trivial.}
\end{remark}

We begin by proving a compactness result {\MMM for sequences of admissible states with equi-bounded energies. In particular, we establish the convergence of compositions of magnetizations with deformations}. Recall the function $\gamma$ introduced in \eqref{eqn:W-coercivity}.

\begin{proposition}[Compactness]
\label{prop:compactness}
Let $(\boldsymbol{q}_n)\subset \mathcal{Q}$ with $\boldsymbol{q}_n=(\boldsymbol{y}_n,\boldsymbol{m}_n)$ satisfy
\begin{equation}
    \label{eqn:compactness-bound}
    ||\nabla \boldsymbol{y}_n||_{L^p(\Omega;\rtt)}\leq C, \qquad ||\nabla \boldsymbol{m}_n||_{L^2(\Omega^{\boldsymbol{y}_n};\rtt)}\leq C, \qquad ||\gamma(\det \nabla \boldsymbol{y}_n)||_{L^1(\Omega)}\leq C
\end{equation}
for every $n \in \N$. Then, there exists $\boldsymbol{q} \in \mathcal{Q}$ with $\boldsymbol{q}=(\boldsymbol{y},\boldsymbol{m})$ such that, up to subsequences, we have $\boldsymbol{q}_n \to \boldsymbol{q}$ in $\mathcal{Q}$ and also
\begin{equation}
    \label{eqn:convergence-composition}
    \text{$\boldsymbol{m}_n \circ \boldsymbol{y}_n \to  \boldsymbol{m}\circ \boldsymbol{y}$ in $L^{a}(\Omega;\R^3)$ for every $1 \leq {a} < \infty$.}
\end{equation}
\end{proposition}

{\MMM 
\begin{remark}[Anisotropy and DMI energies]
The crystalline anisotropy and the asymmetric exchange can be easily included in Theorem \ref{thm:existence-minimizers} without additional difficulties. The corresponding energy terms are desccribed by the functionals $E^{\rm ani}\colon \mathcal{Q}\to \R$ and $E^{\rm DMI}\colon \mathcal{Q}\to \R$ defined, for $\boldsymbol{q}=(\boldsymbol{y},\boldsymbol{m})\in\mathcal{Q}$, by
\begin{equation*}
    E^{\rm ani}(\boldsymbol{q})\coloneqq \int_{\Omega^{\boldsymbol{y}}}\phi(\boldsymbol{m})\,\d\boldsymbol{\xi},\qquad E^{\rm DMI}(\boldsymbol{q})\coloneqq \kappa \int_{\Omega^{\boldsymbol{y}}} \curl\boldsymbol{m}\cdot\boldsymbol{m}\,\d\boldsymbol{\xi},
\end{equation*}
where $\phi\colon \S^2 \to \R$ is continuous and $\kappa \in \R$. These two functionals  are indeed continuous with respect to the convergences given by Proposition \ref{prop:compactness}. The continuity of $E^{\rm DMI}$ is evident from \eqref{eqn:Q-convergence-nabla-m}--\eqref{eqn:Q-convergence-strong}. The continuity of $E^{\rm ani}$ follows easily from \eqref{eqn:Q-convergence-y} and \eqref{eqn:convergence-composition}. By \eqref{eqn:convergence-composition}, we can assume that compositions converge almost everywhere so that, by the Dominated Convergence Theorem, $\phi(\boldsymbol{m}_n\circ\boldsymbol{y}_n)\to\phi(\boldsymbol{m}\circ\boldsymbol{y})$ in $L^a(\Omega)$ for every $1\leq a<\infty$.
Then, exploiting the weak convergence of Jacobian determinants, which follows from \eqref{eqn:Q-convergence-y}, and employing the change of variable formula, we obtain
\begin{equation*}
    E^{\rm ani}(\boldsymbol{q}_n)=\int_\Omega \phi(\boldsymbol{m}_n\circ\boldsymbol{y}_n)\,\det\nabla\boldsymbol{y}_n\d\boldsymbol{x}\to \int_\Omega \phi(\boldsymbol{m}\circ\boldsymbol{y})\,\det\nabla\boldsymbol{y}\,\d\boldsymbol{x}=E^{\rm ani}(\boldsymbol{q}).
\end{equation*}
We mention that the continuity of $E^{\rm ani}$ can be also established without relying on \eqref{eqn:convergence-composition}, but exploiting only \eqref{eqn:Q-convergence-y} and \eqref{eqn:Q-convergence-strong} by means of a localization argument based on \eqref{eqn:inner-approximation}--\eqref{eqn:outer-approximation}.
\end{remark}
}

\begin{proof}[Proof of Proposition \ref{prop:compactness}]
For convenience of the reader, the proof is subdivided into three steps. 

\textbf{Step 1 (Compactness).}
By \eqref{eqn:compactness-bound}, using the Poincaré inequality with boundary terms, we deduce that $(\boldsymbol{y}_n)$ is bounded in $W^{1,p}(\Omega;\R^3)$ . Thus, up to subsequences, { \eqref{eqn:Q-convergence-y} holds} for some $\boldsymbol{y}\in W^{1,p}(\Omega;\R^3)$. 

We claim that $\boldsymbol{y}\in \mathcal{Y}$. Thanks to Remark \ref{rem:ciaret-necas} and the compactness of the trace operator, we only have to prove that $\det \nabla \boldsymbol{y}>0$ almost everywhere in $\Omega$. 
By the weak continuity of Jacobian minors, $\det \nabla \boldsymbol{y}_n \wk \det \nabla \boldsymbol{y}$ in $L^{p/3}(\Omega)$. Then, for every $S \subset \Omega$ measurable, we have
\begin{equation*}
    \int_S \det \nabla \boldsymbol{y}\,\d\boldsymbol{x}=\lim_{n\to \infty} \int_S \det \nabla \boldsymbol{y}_n\,\d \boldsymbol{x}\geq 0,
\end{equation*}
and, given the arbitrariness of $S$, we deduce that $\det \nabla \boldsymbol{y}\geq 0$ almost everywhere in $\Omega$.
By contradiction, suppose that $\det \nabla \boldsymbol{y}=0$ on a measurable set $S_0 \subset \Omega$ with $\leb(S_0)>0$. In this case, up to subsequences, $\det \nabla \boldsymbol{y}_n \to 0$ almost everywhere in $S_0$, and, taking into account \eqref{eqn:W-coercivity}, we obtain $\gamma(\det \nabla \boldsymbol{y}_n)\to +\infty$ almost everywhere in $S_0$. Then, by the Fatou lemma, we obtain $\liminf_{n\to\infty} \int_{S_0} \gamma(\det \nabla \boldsymbol{y}_n)\,\d \boldsymbol{x}=+\infty$, which contradicts \eqref{eqn:compactness-bound}. 
Therefore, $\leb(S_0)=0$ and $\det \nabla \boldsymbol{y}>0$ almost everywhere in $\Omega$.

The compactness of the sequence $(\boldsymbol{q}_n)$ is proved as in \cite[Proposition 2.1]{kruzik.stefanelli.zeman}.
By the Morrey embedding,  we have $\boldsymbol{y}_n\to \boldsymbol{y}$ uniformly in $\Omega$. 
From this, we obtain the following:
\begin{equation}
    \label{eqn:inner-approximation}
    \forall \,A \subset \subset \Omega^{\boldsymbol{y}}\:\text{open}, \quad  A \subset \Omega^{\boldsymbol{y}_n}\quad \text{for $n \gg 1$ depending on $A$,}
\end{equation}
\begin{equation}
    \label{eqn:outer-approximation}
    \forall \,O \supset \supset \Omega^{\boldsymbol{y}}\:\text{open}, \quad  O \supset \Omega^{\boldsymbol{y}_n}\quad \text{for $n \gg 1$ depending on $O$.}
\end{equation}
To see \eqref{eqn:inner-approximation}, let $A \subset \subset \Omega^{\boldsymbol{y}}$ be open so that $\dist(\partial A;\partial \Omega^{\boldsymbol{y}})>0$. Recall that $\partial \Omega^{\boldsymbol{y}}=\boldsymbol{y}(\partial \Omega)$ by Lemma \ref{lem:deformed-configuration}. Then, for $n \gg 1$ depending on $A$, we have
\begin{equation*}
    ||\boldsymbol{y}_n-\boldsymbol{y}||_{C^0(\closure{\Omega};\R^3)}\leq \dist(\partial A;\boldsymbol{y}(\partial \Omega)).
\end{equation*}
Let $\boldsymbol{\xi}\in A$. We obtain
\begin{equation*}
    ||\boldsymbol{y}_n-\boldsymbol{y}||_{C^0(\closure{\Omega};\R^3)}\leq\dist(\boldsymbol{\xi};\boldsymbol{y}(\partial \Omega)),
\end{equation*}
and, by the stability property of the degree \cite[Theorem 2.3, Claim (1)]{fonseca.gangbo},  we deduce $\boldsymbol{\xi}\notin \boldsymbol{y}_n(\partial \Omega)$ and $\deg(\boldsymbol{y}_n,\Omega,\boldsymbol{\xi})=\deg(\boldsymbol{y},\Omega,\boldsymbol{\xi})$ for $n \gg 1$. As $\deg(\boldsymbol{y},\Omega,\boldsymbol{\xi})>0$ by Remark \ref{rem:topological-image},  the solvability property of the degree \cite[Theorem 2.1]{fonseca.gangbo} gives  $\boldsymbol{\xi}\in \Omega^{\boldsymbol{y}_n}$ for $n \gg 1$. This proves \eqref{eqn:inner-approximation}, while \eqref{eqn:outer-approximation} is immediate.

Let $A \subset \subset \Omega^{\boldsymbol{y}}$ be open {\MMM with smooth boundary} and $n \gg 1$ as in \eqref{eqn:inner-approximation}. From  \eqref{eqn:compactness-bound}, we have 
\begin{equation}
    \label{eqn:compacntess-nabla-m}
    \int_A |\nabla \boldsymbol{m}_n|^2\,\d\boldsymbol{\xi}\leq \int_{\Omega^{\boldsymbol{y}_n}} |\nabla \boldsymbol{m}_n|^2\,\d\boldsymbol{\xi}\leq C,
\end{equation}
for every $n \gg1$.
Recalling that magnetizations are sphere-valued, we deduce that $(\boldsymbol{m}_n)$ is bounded in $W^{1,2}(A;\R^3)$, so that, up to subsequences, $\boldsymbol{m}_n \wk \boldsymbol{m}$ in $W^{1,2}(A;\R^3)$ for some $\boldsymbol{m}\in W^{1,2}(A;\R^3)$. By the Rellich embedding, $\boldsymbol{m}_n \to \boldsymbol{m}$ in $L^2(A;\R^3)$ and, in turn, $|\boldsymbol{m}|=1$ almost everywhere in $A$. The map $\boldsymbol{m}\in W^{1,2}_\loc(\Omega^{\boldsymbol{y}};\S^2)$ does not depend on $A$. In particular, as the right-hand side of \eqref{eqn:compacntess-nabla-m} does not depend on $A$, we actually have  $\boldsymbol{m}\in W^{1,2}(\Omega^{\boldsymbol{y}};\S^2)$. Therefore, $\boldsymbol{q}=(\boldsymbol{y},\boldsymbol{m})\in \mathcal{Q}$. Moreover, arguing with a sequence $(A_j)$ of open sets {\MMM with smooth boundaries} such that $A_j \subset \subset A_{j+1}\subset \subset \Omega^{\boldsymbol{y}}$ for every $j \in \N$ and $\Omega^{\boldsymbol{y}}=\bigcup_{j=1}^\infty A_j$, we select a (not relabeled) subsequence of $(\boldsymbol{m}_n)$ such that
\begin{equation}
    \label{eqn:compacntess-m-convergence}
    \forall A \subset \subset \Omega^{\boldsymbol{y}}\:\text{open}, \quad \text{$\boldsymbol{m}_n \wk \boldsymbol{m}$ in $W^{1,2}(A)$, \: $\boldsymbol{m}_n \to \boldsymbol{m}$ almost everywhere in $A$.} 
\end{equation}
We remark that, for every $A \subset \subset \Omega^{\boldsymbol{y}}$ open, the sequence $(\boldsymbol{m}_n)\subset W^{1,2}(A;\S^2)$ is defined only for $n \gg 1$ depending on $A$.

\textbf{Step 2 (Convergence in $\boldsymbol{\mathcal{Q}}$).} In order to prove that $\boldsymbol{q}_n \to \boldsymbol{q}$ in $\mathcal{Q}$, we are left to show \eqref{eqn:Q-convergence-m} and \eqref{eqn:Q-convergence-nabla-m}. To prove the first claim, we consider $\boldsymbol{\varphi}\in L^2(\R^3;\R^3)$. We need to show that
\begin{equation}
    \label{eqn:Q-weak-convergence}
    \lim_{n\to\infty}\int_{\R^3} (\chi_{\Omega^{\boldsymbol{y}_n}}\boldsymbol{m}_n - \chi_{\Omega^{\boldsymbol{y}}}\boldsymbol{m}) \cdot \boldsymbol{\varphi}\,\d\boldsymbol{x}= 0.
\end{equation}
Let $A,O\subset \R^3$  be open such that $A \subset \subset \Omega^{\boldsymbol{y}} \subset \subset O$. We write
\begin{equation}
\label{eqn:Q-convergence-split}
    \begin{split}
        \int_{\R^3} (\chi_{\Omega^{\boldsymbol{y}_n}}\boldsymbol{m}_n - \chi_{\Omega^{\boldsymbol{y}}}\boldsymbol{m}) \cdot \boldsymbol{\varphi}\,\d\boldsymbol{x}&=\int_{A} (\chi_{\Omega^{\boldsymbol{y}_n}}\boldsymbol{m}_n - \chi_{\Omega^{\boldsymbol{y}}}\boldsymbol{m}) \cdot \boldsymbol{\varphi}\,\d\boldsymbol{x}\\
        &+\int_{O \setminus A} (\chi_{\Omega^{\boldsymbol{y}_n}}\boldsymbol{m}_n - \chi_{\Omega^{\boldsymbol{y}}}\boldsymbol{m}) \cdot \boldsymbol{\varphi}\,\d\boldsymbol{x}\\
        &+\int_{\R^3\setminus O} (\chi_{\Omega^{\boldsymbol{y}_n}}\boldsymbol{m}_n - \chi_{\Omega^{\boldsymbol{y}}}\boldsymbol{m}) \cdot \boldsymbol{\varphi}\,\d\boldsymbol{x}. 
    \end{split}
\end{equation}
For the first integral on the right-hand side of \eqref{eqn:Q-convergence-split}, by \eqref{eqn:inner-approximation} for $n\gg 1$  we have
\begin{equation}
    \int_{A} (\chi_{\Omega^{\boldsymbol{y}_n}}\boldsymbol{m}_n - \chi_{\Omega^{\boldsymbol{y}}}\boldsymbol{m}) \cdot \boldsymbol{\varphi}\,\d\boldsymbol{x}= \int_A (\boldsymbol{m}_n - \boldsymbol{m}) \cdot \boldsymbol{\varphi}\,\d\boldsymbol{x},
\end{equation}
where, as $n \to \infty$, the right-hand side goes to zero since $\boldsymbol{m}_n \wk \boldsymbol{m}$ in $W^{1,2}(A;\R^3)$ by \eqref{eqn:compacntess-m-convergence}.
Using the H\"{o}lder inequality, the second integral on the right-hand side of \eqref{eqn:Q-convergence-split} is estimated as follows
\begin{equation}
    \left | \int_{O \setminus A} (\chi_{\Omega^{\boldsymbol{y}_n}}\boldsymbol{m}_n-\chi_{\Omega^{\boldsymbol{y}}}\boldsymbol{m}) \cdot \boldsymbol{\varphi}\,\d \boldsymbol{x} \right | \leq 2 \,\sqrt{\leb(O \setminus A)}\,||\boldsymbol{\varphi}||_{L^2(\R^3;\R^3)}.
\end{equation}
By \eqref{eqn:outer-approximation}, the third integral on the right-hand side of \eqref{eqn:Q-convergence-split} equals zero for $n \gg 1$. Therefore, we obtain
\begin{equation*}
    \limsup_{n\to\infty} \left | \int_{\R^3} (\chi_{\Omega^{\boldsymbol{y}_n}}\boldsymbol{m}_n - \chi_{\Omega^{\boldsymbol{y}}}\boldsymbol{m}) \cdot \boldsymbol{\varphi}\,\d\boldsymbol{x} \right|\leq 2 \,\sqrt{\leb(O \setminus A)}\,||\boldsymbol{\varphi}||_{L^2(\R^3;\R^3)},
\end{equation*}
from which, letting $O \searrow \closure{\Omega^{\boldsymbol{y}}}$ and $A \nearrow \Omega^{\boldsymbol{y}}$ so that $\leb(O \setminus A)\to \leb(\partial \Omega^{\boldsymbol{y}})=0$, we deduce \eqref{eqn:Q-weak-convergence}. Here, we used that $\partial \Omega^{\boldsymbol{y}}=\boldsymbol{y}(\partial \Omega)$ by Lemma \ref{lem:deformed-configuration} and that $\leb(\boldsymbol{y}(\partial \Omega))=0$ thanks to the Lusin property $(N)$. Thus  \eqref{eqn:Q-convergence-m} is proved. 

For the second claim, we proceed in a similar way. Given $\boldsymbol{\Phi}\in L^2(\R^3;\rtt)$, we need to show
\begin{equation}
    \label{eqn:Q-weak-convergence-gradient}
    \lim_{n \to\infty}\int_{\R^3} (\chi_{\Omega^{\boldsymbol{y}_n}}\nabla\boldsymbol{m}_n - \chi_{\Omega^{\boldsymbol{y}}}\nabla\boldsymbol{m}) : \boldsymbol{\Phi}\,\d\boldsymbol{x}= 0.
\end{equation}
As before, we consider $A,O\subset \R^3$ open with $A \subset \subset \Omega^{\boldsymbol{y}}\subset \subset O$ and we write
\begin{equation}
\label{eqn:Q-convergence-nabla-split}
    \begin{split}
        \int_{\R^3} (\chi_{\Omega^{\boldsymbol{y}_n}}\nabla\boldsymbol{m}_n - \chi_{\Omega^{\boldsymbol{y}}}\nabla\boldsymbol{m}) : \boldsymbol{\Phi}\,\d\boldsymbol{x}&=\int_{A} (\chi_{\Omega^{\boldsymbol{y}_n}}\nabla\boldsymbol{m}_n - \chi_{\Omega^{\boldsymbol{y}}}\nabla\boldsymbol{m}) : \boldsymbol{\Phi}\,\d\boldsymbol{x}\\
        &+\int_{O \setminus A} (\chi_{\Omega^{\boldsymbol{y}_n}}\nabla\boldsymbol{m}_n - \chi_{\Omega^{\boldsymbol{y}}}\nabla\boldsymbol{m}) : \boldsymbol{\Phi}\,\d\boldsymbol{x}\\
        &+\int_{\R^3\setminus O} (\chi_{\Omega^{\boldsymbol{y}_n}}\nabla\boldsymbol{m}_n - \chi_{\Omega^{\boldsymbol{y}}}\nabla\boldsymbol{m}) : \boldsymbol{\Phi}\,\d\boldsymbol{x}. 
    \end{split}
\end{equation}
For the first integral on the right-hand side of \eqref{eqn:Q-convergence-nabla-split}, by \eqref{eqn:inner-approximation}, for $n \gg 1$ we have
\begin{equation*}
    \int_{A} (\chi_{\Omega^{\boldsymbol{y}_n}}\nabla\boldsymbol{m}_n - \chi_{\Omega^{\boldsymbol{y}}}\nabla\boldsymbol{m}) : \boldsymbol{\Phi}\,\d\boldsymbol{x}=\int_A (\nabla \boldsymbol{m}_n-\nabla \boldsymbol{m}):\boldsymbol{\Phi}\,\d\boldsymbol{x},
\end{equation*}
and, as $n \to \infty$, the right-hand side goes to zero since $\boldsymbol{m}_n \wk \boldsymbol{m}$ in $W^{1,2}(A;\R^3)$ by \eqref{eqn:compacntess-m-convergence}. Note that the sequence $(\chi_{\Omega^{\boldsymbol{y}_n}}\nabla \boldsymbol{m}_n)\subset L^2(\R^3;\rtt)$ is bounded by \eqref{eqn:compactness-bound}. Using the H\"{o}lder inequality, the second integral on the right-hand side of \eqref{eqn:Q-convergence-nabla-split} is estimated as follows:
\begin{equation*}
    \begin{split}
    \Bigg | \int_{O \setminus A} (\chi_{\Omega^{\boldsymbol{y}_n}}&\nabla\boldsymbol{m}_n - \chi_{\Omega^{\boldsymbol{y}}}\nabla\boldsymbol{m}) : \boldsymbol{\Phi}\,\d\boldsymbol{x} \Bigg |\\
    &\leq \left (||\chi_{\Omega^{\boldsymbol{y}_n}}\nabla\boldsymbol{m}_n||_{L^2(\R^3;\rtt)}+||\chi_{\Omega^{\boldsymbol{y}}}\nabla\boldsymbol{m}||_{L^2(\R^3;\rtt)} \right )\,||\boldsymbol{\Phi}||_{L^2(O \setminus A;\rtt)}\\
    &\leq \left (C +||\chi_{\Omega^{\boldsymbol{y}}}\nabla\boldsymbol{m}||_{L^2(\R^3;\rtt)}\right )\,||\boldsymbol{\Phi}||_{L^2(O \setminus A;\rtt)}.
    \end{split}
\end{equation*}
By \eqref{eqn:outer-approximation}, the third integral on the right-hand side of \eqref{eqn:Q-convergence-nabla-split} equals zero for $n \gg 1$. Therefore, we obtain
\begin{equation*}
    \limsup_{n \to \infty} \left | \int_{\R^3} (\chi_{\Omega^{\boldsymbol{y}_n}}\nabla\boldsymbol{m}_n - \chi_{\Omega^{\boldsymbol{y}}}\nabla\boldsymbol{m}) : \boldsymbol{\Phi}\,\d\boldsymbol{x} \right |\leq \left (C +||\chi_{\Omega^{\boldsymbol{y}}}\nabla\boldsymbol{m}||_{L^2(\R^3;\rtt)}\right )\,||\boldsymbol{\Phi}||_{L^2(O \setminus A;\rtt)}.
\end{equation*}
From this, letting $O \searrow \closure{\Omega^{\boldsymbol{y}}}$ and $A \nearrow \Omega^{\boldsymbol{y}}$ so that $\leb(O \setminus A)\to \leb(\partial \Omega^{\boldsymbol{y}})=0$ and, in turn, $||\boldsymbol{\Phi}||_{L^2(O \setminus A;\rtt)}\to 0$, we deduce \eqref{eqn:Q-weak-convergence-gradient}. Thus also  \eqref{eqn:Q-convergence-nabla-m} is proved.

\textbf{Step 3 (Convergence of the compositions).} By Proposition \ref{prop:global-invertibility}, $\boldsymbol{y}_n^{-1}\in W^{1,1}(\Omega^{\boldsymbol{y}_n};\R^3)$ with $\det \nabla \boldsymbol{y}_n^{-1}\in L^1(\Omega^{\boldsymbol{y}_n})$ for every $n \in \N$. {\MMM We claim that, for every open set $A \subset \subset \Omega^{\boldsymbol{y}}$, the sequence $(\det \nabla \boldsymbol{y}_n^{-1})\subset L^1(A)$ is equi-integrable}. To show this, we argue as in \cite[Proposition 7.8]{barchiesi.henao.moracorral}. Define $\widehat{\gamma}\colon (0,+\infty)\to [0,+\infty)$ by setting $\widehat{\gamma}({\MMM z})\coloneqq {\MMM z}\,\gamma(1/{\MMM z})$. In this case
\begin{equation*}
    \lim_{{\MMM z} \to +\infty} \frac{\widehat{\gamma}({\MMM z})}{{\MMM z}}=\lim_{{\MMM z} \to +\infty} \gamma(1/{\MMM z})=\lim_{h \to 0^+} \gamma(h)=+\infty,
\end{equation*}
where we used \eqref{eqn:W-coercivity}.
Using the change-of-variable formula, we compute
\begin{equation*}
    \begin{split}
        \int_{\Omega^{\boldsymbol{y}_n}} \widehat{\gamma}(\det \nabla \boldsymbol{y}_n^{-1})\,\d\boldsymbol{\xi}&=\int_{\Omega^{\boldsymbol{y}_n}} \gamma(1/\det \nabla \boldsymbol{y}_n^{-1})\,\det \nabla \boldsymbol{y}_n^{-1}\,\d\boldsymbol{\xi}\\
        &=\int_{\Omega^{\boldsymbol{y}_n}} \gamma(\det \nabla \boldsymbol{y}_n)\circ \boldsymbol{y}_n^{-1}\,(\det \nabla \boldsymbol{y}_n)^{-1}\circ \boldsymbol{y}_n^{-1}\,\d\boldsymbol{\xi}\\
        &=\int_\Omega \gamma(\det \nabla \boldsymbol{y}_n)\,\d\boldsymbol{x},
    \end{split}
\end{equation*}
where the right-hand side is uniformly bounded by \eqref{eqn:compactness-bound}. Thus, {\MMM the claim follows by the de la Vallée-Poussin Criterion \cite[Theorem 2.29]{fonseca.leoni}}. In particular, using the area formula as in Remark \ref{rem:area-formula-inverse}, we deduce the following: 
\begin{equation}
    \label{eqn:equi-integrability}
    \begin{split}
        &\hspace*{4mm}\text{\MMM for every $A \subset \subset \Omega^{\boldsymbol{y}}$ open and for every $\varepsilon>0$ there exists \MMM $\delta(A,\varepsilon)>0$ such that}\\
        &\text{\MMM for every $F \subset A$ measurable with $\leb(F)<\delta(A,\varepsilon)$ there holds $\sup_{n \in \N}\leb(\boldsymbol{y}_n^{-1}(F))<\varepsilon$.}
    \end{split}
\end{equation}

We now prove that $\boldsymbol{m}_n \circ \boldsymbol{y}_n \to \boldsymbol{m}\circ \boldsymbol{y}$ in $L^1(\Omega;\R^3)$. Fix {\MMM $\varepsilon>0$}. Take $A \subset \subset \Omega^{\boldsymbol{y}}$ open such that {\MMM$\leb(\Omega \setminus \boldsymbol{y}^{-1}(A))<\varepsilon$}. We compute
\begin{equation}
    \label{eqn:composition1}
    \int_\Omega |\boldsymbol{m}_n \circ \boldsymbol{y}_n - \boldsymbol{m}\circ \boldsymbol{y}|\,\d\boldsymbol{x}=\int_{\Omega \setminus \boldsymbol{y}^{-1}(A)} |\boldsymbol{m}_n \circ \boldsymbol{y}_n - \boldsymbol{m}\circ \boldsymbol{y}|\,\d\boldsymbol{x} + \int_{\boldsymbol{y}^{-1}(A)} |\boldsymbol{m}_n \circ \boldsymbol{y}_n - \boldsymbol{m}\circ \boldsymbol{y}|\,\d\boldsymbol{x}.
\end{equation}
As magnetizations are sphere-valued, {\MMM for every $n\in \N$} the first integral on the right-hand side of \eqref{eqn:composition1} is bounded by $2 \leb(\Omega \setminus \boldsymbol{y}^{-1}(A))<2\varepsilon$. For the second integral on the right-hand side of \eqref{eqn:composition1}, we split it as
\begin{equation}
    \label{eqn:composition2}
    \begin{split}
        \int_{\boldsymbol{y}^{-1}(A)} |\boldsymbol{m}_n \circ \boldsymbol{y}_n - \boldsymbol{m}\circ \boldsymbol{y}|\,\d\boldsymbol{x}&=\int_{\boldsymbol{y}^{-1}(A) \setminus \boldsymbol{y}_n^{-1}(A)} |\boldsymbol{m}_n \circ \boldsymbol{y}_n - \boldsymbol{m}\circ \boldsymbol{y}|\,\d\boldsymbol{x}\\
        &+\int_{\boldsymbol{y}^{-1}(A)\cap \boldsymbol{y}_n^{-1}(A)} |\boldsymbol{m}_n \circ \boldsymbol{y}_n - \boldsymbol{m}\circ \boldsymbol{y}|\,\d\boldsymbol{x}.
    \end{split}
\end{equation}
We claim that {\MMM $\leb(\boldsymbol{y}^{-1}(A)\setminus \boldsymbol{y}_n^{-1}(A))<\varepsilon$ for $n\gg 1$ depending only on $\varepsilon$, so that the second integral on the right-hand side of \eqref{eqn:composition2} is bounded by $2\varepsilon$}. To see this, let $V\subset \R^3$ be open and such that $A \subset \subset V \subset \subset \Omega^{\boldsymbol{y}}$. In this case, $\boldsymbol{y}(\boldsymbol{y}^{-1}(A))=A \subset \subset V$ so that, by uniform convergence, $\boldsymbol{y}_n(\boldsymbol{y}^{-1}(A))\subset V$ for $n \gg 1$ which, in turn, gives $\boldsymbol{y}^{-1}(A)\subset \boldsymbol{y}_n^{-1}(V)$ for $n \gg 1$. Then, we have
\begin{equation}
    \label{eqn:preimage-measure-convergence}
    \boldsymbol{y}^{-1}(A)\setminus \boldsymbol{y}_n^{-1}(A)\subset \boldsymbol{y}_n^{-1}(V)\setminus \boldsymbol{y}_n^{-1}(A) =\boldsymbol{y}_n^{-1}(V \setminus A),
\end{equation}
for $n \gg1$. In particular, if $V$ is chosen such that  $\leb(V \setminus A)<{\MMM \delta(V,\varepsilon)}$ with ${\MMM \delta(V,\varepsilon)}>0$ given by \eqref{eqn:equi-integrability}. Hence, for $n \gg 1$ depending only on $\varepsilon$, from \eqref{eqn:equi-integrability} and \eqref{eqn:preimage-measure-convergence}, we obtain $\leb(\boldsymbol{y}^{-1}(A)\setminus \boldsymbol{y}_n^{-1}(A))<\varepsilon$ and the claim is proved. 

To estimate the second integral on the right-hand side of \eqref{eqn:composition2} we proceed as follows. {\MMM Henceforth, we will simply write $\delta$ in place of $\delta(A,\varepsilon)$, where $\delta(A,\varepsilon)>0$ is given by \eqref{eqn:equi-integrability}. Without loss of generality, we can assume that $\delta$ is sufficiently small in order to have $\leb(\boldsymbol{y}^{-1}(F))<\varepsilon$ for every $F \subset A$ measurable with $\leb(F)<\delta$.
By the Lusin Theorem, there exists $K_1 \subset A$ compact with $\leb(A\setminus K_1)<\delta/2$ such that $\boldsymbol{m}\restr{K_1}$ is continuous while, by the Egorov Theorem, there exists  $K_2 \subset A$ compact with $\leb(A\setminus K_2)<\delta/2$ such that $\boldsymbol{m}_n \to \boldsymbol{m}$ uniformly on $K_2$. Set $K \coloneqq K_1 \cap K_2$, so that $K \subset A$ is compact and $\leb(A \setminus K)<\delta$. We have 
\begin{equation*}
    \begin{split}
        \boldsymbol{y}^{-1}(A) \cap \boldsymbol{y}_n^{-1}(A)\subset  \left(\boldsymbol{y}^{-1}(K)\cap \boldsymbol{y}_n^{-1}(K) \right) \cup \boldsymbol{y}^{-1}(A\setminus K)
        \cup \boldsymbol{y}_n^{-1}(A \setminus K)
    \end{split}
\end{equation*}
so that we estimate the second integral on the right-hand side of \eqref{eqn:composition2} as follows
\begin{equation}
    \label{eqn:composition3}
    \begin{split}
        \int_{\boldsymbol{y}^{-1}(A)\cap \boldsymbol{y}_n^{-1}(A)} |\boldsymbol{m}_n \circ \boldsymbol{y}_n - \boldsymbol{m}\circ \boldsymbol{y}|\,\d\boldsymbol{x}&\leq \int_{\boldsymbol{y}^{-1}(K)\cap \boldsymbol{y}_n^{-1}(K)} |\boldsymbol{m}_n \circ \boldsymbol{y}_n - \boldsymbol{m}\circ \boldsymbol{y}|\,\d\boldsymbol{x}\\
        &+ \int_{\boldsymbol{y}^{-1}(A\setminus K)
        \cup \boldsymbol{y}_n^{-1}(A \setminus K)} |\boldsymbol{m}_n \circ \boldsymbol{y}_n - \boldsymbol{m}\circ \boldsymbol{y}|\,\d\boldsymbol{x}.
    \end{split}
\end{equation}
The second integral on the right-hand side of \eqref{eqn:composition3}, we have
\begin{equation}
\label{eqn:composition4}
    \begin{split}
        \int_{\boldsymbol{y}^{-1}(A\setminus K)\cup \boldsymbol{y}_n^{-1}(A\setminus K)} |\boldsymbol{m}_n \circ \boldsymbol{y}_n - \boldsymbol{m}\circ \boldsymbol{y}|\,\d\boldsymbol{x} &\leq 2 \leb \left (\boldsymbol{y}^{-1}(A\setminus K)\cup \boldsymbol{y}_n^{-1}(A\setminus K)\right )\\
        &\leq 2 \left (\leb(\boldsymbol{y}^{-1}(A\setminus K))+ \leb(\boldsymbol{y}_n^{-1}(A\setminus K)\right)\\
        &< 4 \varepsilon ,
    \end{split}
\end{equation}
where, in the last line, we used \eqref{eqn:equi-integrability}.
For the first integral on the right-hand side of \eqref{eqn:composition3}, we have
\begin{equation}
\label{eqn:composition5}
    \begin{split}
        \int_{\boldsymbol{y}^{-1}(K)\cap \boldsymbol{y}_n^{-1}(K)} |\boldsymbol{m}_n \circ \boldsymbol{y}_n - \boldsymbol{m}\circ \boldsymbol{y}|\,\d\boldsymbol{x}&\leq \int_{\boldsymbol{y}^{-1}(K)\cap \boldsymbol{y}_n^{-1}(K)} |\boldsymbol{m}_n \circ \boldsymbol{y}_n - \boldsymbol{m}\circ \boldsymbol{y}_n|\,\d\boldsymbol{x}\\
        &+\int_{\boldsymbol{y}^{-1}(K)\cap \boldsymbol{y}_n^{-1}(K)} |\boldsymbol{m} \circ \boldsymbol{y}_n - \boldsymbol{m}\circ \boldsymbol{y}|\,\d\boldsymbol{x}.
    \end{split}
\end{equation}
Note that, in the previous equation, the composition $\boldsymbol{m}\circ \boldsymbol{y}_n$ is meaningful, at least for $n \gg 1$, since the domain of integration is a subset of $\boldsymbol{y}^{-1}(A)\cap \boldsymbol{y}_n^{-1}(A)$ and $A \subset \Omega^{\boldsymbol{y}_n}$.
Given the choice of $K$, for $n \gg 1$ depending only on $\varepsilon$, we have $\sup_{K} |\boldsymbol{m}_n-\boldsymbol{m}|<\varepsilon/\leb(\Omega)$. Thus, we estimate
\begin{equation*}
    \int_{\boldsymbol{y}^{-1}(K)\cap \boldsymbol{y}_n^{-1}(K)} |\boldsymbol{m}_n \circ \boldsymbol{y}_n - \boldsymbol{m}\circ \boldsymbol{y}_n|\,\d\boldsymbol{x}< \frac{\varepsilon}{\leb(\Omega)}\,\leb(\boldsymbol{y}^{-1}(K)\cap \boldsymbol{y}_n^{-1}(K))<\varepsilon.
\end{equation*}
On the other hand, $\boldsymbol{m}$ is uniformly continuous on $K$. Hence, there exists $\bar{\delta}(\varepsilon)>0$ such that for every $\boldsymbol{\xi}_1,\boldsymbol{\xi}_2\in K$ with $|\boldsymbol{\xi}_1-\boldsymbol{\xi}_2|<\bar{\delta}(\varepsilon)$ there holds $|\boldsymbol{m}(\xi_1)-\boldsymbol{m}(\boldsymbol{\xi}_2)|<\varepsilon/\leb(\Omega)$. As a consequence, for $n\gg 1$ such that $||\boldsymbol{y}_n-\boldsymbol{y}||_{C^0(\closure{\Omega};\R^3)}<\bar{\delta}(\varepsilon)$, we obtain
\begin{equation*}
    \int_{\boldsymbol{y}^{-1}(K)\cap \boldsymbol{y}_n^{-1}(K)} |\boldsymbol{m} \circ \boldsymbol{y}_n - \boldsymbol{m}\circ \boldsymbol{y}|\,\d\boldsymbol{x}<\frac{\varepsilon}{\leb(  \Omega)}\,\leb(\boldsymbol{y}^{-1}(K)\cap \boldsymbol{y}_n^{-1}(K))<\varepsilon.
\end{equation*}
Therefore, combining \eqref{eqn:composition1}--\eqref{eqn:composition2} and \eqref{eqn:composition3}--\eqref{eqn:composition5}, we deduce that
\begin{equation*}
    \limsup_{n\to\infty} \int_\Omega |\boldsymbol{m}_n \circ \boldsymbol{y}_n - \boldsymbol{m}\circ \boldsymbol{y}|\,\d\boldsymbol{x} \leq 10\varepsilon.
\end{equation*}
As $\varepsilon>0$ was arbitrary, this concludes the proof of the convergence of compositions in $L^1(\Omega;\R^3)$. The convergence in $L^a(\Omega;\R^3)$ for every $1<a<\infty$ follows immediately by extracting a subsequence that converges almost everywhere and by applying the Dominated Convergence Theorem.}
\end{proof}

We are now ready to prove Theorem \ref{thm:existence-minimizers}.

\begin{proof}[Proof of Theorem \ref{thm:existence-minimizers}]

Let $(\boldsymbol{q}_n)\subset \mathcal{Q}$ with $\boldsymbol{q}_n=(\boldsymbol{y}_n,\boldsymbol{m}_n)$ be a minimizing sequence for $E$, namely such that $E(\boldsymbol{q}_n)\to \inf_{\mathcal{Q}}E$. In particular, $\sup_{n\in\N}E(\boldsymbol{q}_n)<+\infty$. From \eqref{eqn:W-coercivity}, we deduce \eqref{eqn:compactness-bound} so that we can apply Proposition \ref{prop:compactness}. This gives a (not relabeled) subsequence $(\boldsymbol{q}_n)$ and an admissible state $\boldsymbol{q}=(\boldsymbol{y},\boldsymbol{m})\in \mathcal{Q}$ such that $\boldsymbol{q}_{n}\to \boldsymbol{q}$ in $\mathcal{Q}$ and $\boldsymbol{m}_{n}\circ \boldsymbol{y}_{n}\to \boldsymbol{m}\circ \boldsymbol{y}$ in $L^{ a}(\Omega;\R^3)$ for every $1 \leq {a}<\infty$.

We claim that
\begin{equation}
    \label{eqn:existence-lsc}
    E(\boldsymbol{q})\leq \liminf_{n\to \infty} E(\boldsymbol{q}_{n}),
\end{equation}
so that $\boldsymbol{q}$ is a minimizer of $E$.
We focus on the elastic energy first.
We have $\nabla \boldsymbol{y}_{n} \wk \nabla \boldsymbol{y}$ in $L^p(\Omega;\rtt)$ and, by the weak continuity of Jacobian minors, we also have $\cof\,\nabla \boldsymbol{y}_{n} \wk \cof\,\nabla \boldsymbol{y}$ in $L^{p/2}(\Omega;\rtt)$ and $\det \nabla \boldsymbol{y}_{n} \wk \det \nabla \boldsymbol{y}$ in $L^{p/3}(\Omega)$. Moreover, the subsequence can be chosen in order to have $\boldsymbol{m}_{n}\circ \boldsymbol{y}_{n}\to \boldsymbol{m}\circ \boldsymbol{y}$ almost everywhere in $\Omega$. Thus, given \eqref{eqn:W-polyconvex}, applying \cite[Theorem 5.4]{ball.currie.olver} we prove that
\begin{equation}
    \label{eqn:existence-lsc-el}
    E^{\text{el}}(\boldsymbol{q})\leq \liminf_{n\to\infty} E^{\text{el}}(\boldsymbol{q}_{n}).
\end{equation}
The lower semicontinuity of the exchange energy is immediate. Indeed, by \eqref{eqn:Q-convergence-nabla-m} and the lower semicontinuity of the norm, there holds
\begin{equation}
    \label{eqn:existence-lsc-exc}
    E^{\text{exc}}(\boldsymbol{q})\leq \liminf_{n\to\infty} E^{\text{exc}}(\boldsymbol{q}_{n}).
\end{equation}

We focus on the magnetostatic energy. Denote by $\zeta_n$ a weak solutions of the Maxwell equation corresponding to $\boldsymbol{q}_{n}$. Thus, for every $n \in \N$  and for every $\varphi \in V^{1,2}(\R^3)$, there holds
\begin{equation}
    \label{eqn:existence-maxwell}
    \int_{\R^3} \nabla \zeta_n \cdot \nabla \varphi\,\d\boldsymbol{\xi}=\int_{\R^3} \chi_{\Omega^{\boldsymbol{y}_{n}}}\boldsymbol{m}_{n}\cdot \nabla \varphi\,\d\boldsymbol{\xi}.
\end{equation}
Denote by $V^{1,2}(\R^3)/\R$ the quotient of $V^{1,2}(\R^3)$ with respect to constant functions and recall that this is an Hilbert space with inner product given by
\begin{equation*}
    ([\varphi],[\psi])\mapsto \int_{\R^3} \nabla \varphi \cdot \nabla \psi\,\d\boldsymbol{\xi}.
\end{equation*}
Testing \eqref{eqn:existence-maxwell} with $\varphi=\zeta_n$ and using that $\sup_{n\in\N}\||\chi_{\Omega^{\boldsymbol{y}_{n}}}\boldsymbol{m}_{n}||_{L^2(\R^3;\R^3)}<+\infty$ by \eqref{eqn:Q-convergence-m}, we obtain that $$\sup_{n\in \N}||[\zeta_n]||_{V^{1,2}(\R^3)/\R}=\sup_{n\in \N}||\nabla \zeta_n||_{L^2(\R^3;\R^3)}<+\infty.$$ Therefore, there exists $\zeta \in V^{1,2}(\R^3)$ such that, up to subsequences, we have $[\zeta_n]\wk[\zeta]$ in $V^{1,2}(\R^3)/\R$, or equivalently, $\nabla \zeta_n \wk \nabla \zeta$ in $L^2(\R^3;\R^3)$. Passing to the limit, as $n \to \infty$, in \eqref{eqn:existence-maxwell}, we obtain that
\begin{equation*}
    \int_{\R^3} \nabla \zeta \cdot \nabla \varphi\,\d\boldsymbol{\xi}=\int_{\R^3} \chi_{\Omega^{\boldsymbol{y}}}\boldsymbol{m}\cdot \nabla \varphi\,\d\boldsymbol{\xi},
\end{equation*}
for every $\varphi\in V^{1,2}(\R)$. Thus $\zeta$ is a weak solution of the Maxwell equation corresponding to $\boldsymbol{q}$ and, in turn, $E^{\text{mag}}(\boldsymbol{q})=({\mu_0}/{2})\,||\nabla \zeta||_{L^2(\R^3;\R^3)}^2$. By the lower semicontinuity of the norm, we conclude
\begin{equation}
    \label{eqn:existence-lsc-mag}
    E^{\text{mag}}(\boldsymbol{q})\leq \liminf_{n\to \infty} E^{\text{mag}}(\boldsymbol{q}_{n}).
\end{equation}
Finally, combining \eqref{eqn:existence-lsc-el}-\eqref{eqn:existence-lsc-exc} and \eqref{eqn:existence-lsc-mag}, we get \eqref{eqn:existence-lsc}.
\end{proof}

\section{Quasistatic Setting}
\label{sec:quasistatic}

\subsection{General setting}
\label{subs:basic}
In this section we {\MMM study quasistatic evolutions of the model} driven by time-dependent applied loads and dissipative effects. The framework is the theory of \emph{rate-independent processes} \cite{mielke.roubicek} with the notion of \emph{energetic solutions}.

We start describing the general setting.
The applied loads are determined by the functions 
\begin{equation}
    \label{eqn:applied-loads-regularity}
    \boldsymbol{f}\in C^1([0,T];L^{p'}(\Omega;\R^3)), \quad \boldsymbol{g}\in C^1([0,T];L^{p'}(\Sigma;\R^3)), \quad \boldsymbol{h}\in C^1([0,T];L^2(\R^3;\R^3)),
\end{equation}
where {\MMM $\Sigma \subset \partial \Omega$ is $\haus$-measurable and such that $\Gamma \cap \Sigma=\emptyset$}, representing external body forces, surface forces and magnetic fields, respectively. Define the functional $\mathcal{L}\colon [0,T]\times \mathcal{Q}\to \R$ by setting
\begin{equation}
    \label{eqn:applied-loads-functional}
    \mathcal{L}(t,\boldsymbol{q})\coloneqq \int_\Omega \boldsymbol{f}(t)\cdot \boldsymbol{y}\,\d \boldsymbol{x}+\int_\Sigma \boldsymbol{g}(t)\cdot \boldsymbol{y}\,\d \haus+\int_{\Omega^{\boldsymbol{y}}} \boldsymbol{h}(t)\cdot \boldsymbol{m}\,\d\boldsymbol{\xi},
\end{equation}
where $\boldsymbol{q}=(\boldsymbol{y},\boldsymbol{m})$. The total energy is given by the functional $\mathcal{E}\colon [0,T]\times \mathcal{Q}\to \R$ defined by
\begin{equation}
    \label{eqn:total-energy}
    \mathcal{E}(t,\boldsymbol{q})\coloneqq E(\boldsymbol{q})-\mathcal{L}(t,\boldsymbol{q}).
\end{equation}
 By a repeated application of the H\"{o}lder inequality and the Young inequality and using \eqref{eqn:W-coercivity}, we prove  
\begin{equation}
    \label{eqn:total-energy-coercivity}
    \mathcal{E}(t,\boldsymbol{q})\geq  C_0 ||\nabla \boldsymbol{y}||_{L^p(\Omega;\rtt)}^p+||\gamma(\det \nabla \boldsymbol{y})||_{L^1(\Omega)}+\alpha||\nabla \boldsymbol{m}||^2_{L^2(\Omega^{\boldsymbol{y}};\rtt)}-C_1
\end{equation}
for every $\boldsymbol{q}=(\boldsymbol{y},\boldsymbol{m})\in \mathcal{Q}$. Here, $C_0(K)>0$ and
    $C_1(p,\overline{M},M_{\boldsymbol{f}},M_{\boldsymbol{g}},M_{\boldsymbol{h}})>0$ are two constants,
where $K>0$ was introduced in \eqref{eqn:W-coercivity} and $\overline{M}\coloneqq ||\overline{\boldsymbol{y}}||_{L^{p'}(\Sigma;\R^3)}$ takes into account the boundary datum in \eqref{eqn:admissible-deformations}. Also, we set
\begin{equation*}
    M_{\boldsymbol{f}}\coloneqq ||\boldsymbol{f}||_{C^0([0,T];L^{p'}(\Omega;\R^3))}, \quad M_{\boldsymbol{g}}\coloneqq ||\boldsymbol{g}||_{C^0([0,T];L^{p'}(\Sigma;\R^3))}, \quad M_{\boldsymbol{h}}\coloneqq ||\boldsymbol{h}||_{C^0([0,T];L^{2}(\R^3;\R^3))}.
\end{equation*}
Note that, from \eqref{eqn:total-energy-coercivity}, we deduce $\inf_{[0,T]\times \mathcal{Q}}\mathcal{E}\geq -\, C_1$.

Given the regularity of the applied loads, for every $\boldsymbol{q}=(\boldsymbol{y},\boldsymbol{m})\in \mathcal{Q}$, the map $t \mapsto \mathcal{L}(t,\boldsymbol{q})$ belongs to $C^1([0,T])$. In particular, for every $t \in [0,T]$, we compute
\begin{equation}
    \partial_t \mathcal{E}(t,\boldsymbol{q})=-\partial_t \mathcal{L}(t,\boldsymbol{q})=-\int_\Omega \dot{\boldsymbol{f}}(t)\cdot \boldsymbol{y}\,\d \boldsymbol{x}-\int_\Sigma \dot{\boldsymbol{g}}(t)\cdot \boldsymbol{y}\,\d \haus-\int_{\Omega^{\boldsymbol{y}}} \dot{\boldsymbol{h}}(t)\cdot \boldsymbol{m}\,\d\boldsymbol{\xi}.
\end{equation}
Employing again the H\"{o}lder inequality and the Young inequality and exploiting \eqref{eqn:total-energy-coercivity}, we prove the estimate
\begin{equation}
\label{eqn:time-derivative-total-energy-estimate}
    |\partial_t \mathcal{E}(t,\boldsymbol{q})|\leq L\,(\mathcal{E}(t,\boldsymbol{q})+M).
\end{equation}
Here, $L(p,K,\overline{M}, L_{\boldsymbol{f}},L_{\boldsymbol{g}},L_{\boldsymbol{h}})>0$ and $M(p,K,\overline{M},M_{\boldsymbol{f}},M_{\boldsymbol{g}},M_{\boldsymbol{h}})>0$ are two constants and we set
\begin{equation*}
    L_{\boldsymbol{f}}\coloneqq ||\dot{\boldsymbol{f}}||_{C^0([0,T];L^{p'}(\Omega;\R^3))}, \quad L_{\boldsymbol{g}}\coloneqq ||\dot{\boldsymbol{g}}||_{C^0([0,T];L^{p'}(\Sigma;\R^3))}, \quad  L_{\boldsymbol{h}}\coloneqq ||\dot{\boldsymbol{h}}||_{C^0([0,T];L^{2}(\R^3;\R^3))}.
\end{equation*}
From this, using the Gronwall inequality, we obtain
\begin{equation}
    \label{eqn:total-energy-gronwall}
    \mathcal{E}(t,\boldsymbol{q})+M \leq (\mathcal{E}(s,\boldsymbol{q})+M) e^{L(t-s)},
\end{equation}
for every $\boldsymbol{q}\in \mathcal{Q}$ and $s,t\in [0,T]$ with $s<t$.

As in \cite{roubicek.tomassetti}, we introduce the \emph{Lagrangian magnetization} given, for $\boldsymbol{q}=(\boldsymbol{y},\boldsymbol{m})\in \mathcal{Q}$, by
\begin{equation}
    \label{eqn:Lagrangean-magnetization}
    \mathcal{Z}(\boldsymbol{q})\coloneqq (\adj \nabla \boldsymbol{y}) \, \boldsymbol{m}\circ \boldsymbol{y}.
\end{equation}
The dissipation distance $\mathcal{D}\colon \mathcal{Q}\times \mathcal{Q}\to [0,+\infty)$ is defined as
\begin{equation}
    \label{eqn:dissipation-distance}
    \mathcal{D}(\boldsymbol{q},\widehat{\boldsymbol{q}})\coloneqq \int_\Omega |\mathcal{Z}(\boldsymbol{q})-\mathcal{Z}(\widehat{\boldsymbol{q}})|\,\d \boldsymbol{x}.
\end{equation}
Moreover, the variation of any map $\boldsymbol{q}\colon [0,T]\to \mathcal{Q}$ with respect to $\mathcal{D}$  on the interval $[s,t]\subset [0,T]$ is defined by 
\begin{equation}
    \label{eqn:dissipation-variation}
    \var_{\mathcal{D}}(\boldsymbol{q};[s,t])\coloneqq \sup \left \{ \sum_{i=1}^{N} \mathcal{D}(\boldsymbol{q}(t_i),\boldsymbol{q}(t_{i-1})):\: \Pi=(t_0,\dots,t_N)\:\text{partition of } [s,t]\right \}.
\end{equation}
Here, by a partition of the interval $[s,t]$ we mean any finite ordered set  $\Pi=(t_0,\dots,t_N)\subset [0,T]^N$ with $s=t_0<t_1<\dots<t_{N}=t$. Note that in \eqref{eqn:dissipation-variation} each partition can have different cardinality.

\begin{remark}[Regularity of the applied loads]
The regularity assumptions on the applied loads in \eqref{eqn:applied-loads-regularity} can be relaxed. Indeed, following \cite{mielke.roubicek}, all the analysis can still be carried out if we just assume
\begin{equation*}
     \boldsymbol{f}\in W^{1,1}(0,T;L^{p'}(\Omega;\R^3)), \quad \boldsymbol{g}\in W^{1,1}(0,T;L^{p'}(\Sigma;\R^3)), \quad \boldsymbol{h}\in W^{1,1}(0,T;L^2(\R^3;\R^3)).
\end{equation*}
\end{remark}

{\MMM
\begin{remark}[Time-dependent boundary conditions]
At the current stage, we are not able to treat time-dependent Dirichlet boundary conditions (except for the case in which the boundary datum is given time-by-time by a rigid motion). In particular, the strategy devised in \cite{francfort.mielke} is hindered by the fact that the magnetostatic energy is not differentiable in time.  
However, time-dependent Dirichlet boundary conditions can be included in the analysis in a relaxed form by removing the boundary condition in \eqref{eqn:admissible-deformations} and by enriching the total energy with the term
\begin{equation*}
    \boldsymbol{q}\mapsto \int_{\Gamma} |\boldsymbol{y}-\overline{\boldsymbol{y}}(t)|\,\d\haus,
\end{equation*}
where $\boldsymbol{q}=(\boldsymbol{y},\boldsymbol{m})\in\mathcal{Q}$ and $\overline{\boldsymbol{y}}\in C^1([0,T];C^0(\Gamma;\R^3))$ (or, more generally,  $\overline{\boldsymbol{y}}\in W^{1,1}(0,T;C^0(\Gamma;\R^3))$). From a modeling point of view, in the case in which the material is clamped, this corresponds to keeping track also of deformations of the clamp itself. Additionally, under such relaxed boundary deformations, existence of admissible deformations with finite energy is automatically guaranteed.
\end{remark}
}

The existence of energetic solutions is usually proved in two steps: first, {\MMM  for a given partition of the time interval}, one constructs a time-discrete solution {\MMM by solving the corresponding  incremental minimization problem; then, one considers the piecewise constant interpolants determined by the time-discrete solutions for a sequence of partitions of vanishing size and, by means of compactness arguments, obtains the desired time-continuous solution.}

{\MMM The first step is addressed by employing the results of Section \ref{sec:static}.} Let $\Pi=(t_0,\dots,t_N)$ be a partition of $[0,T]$. We consider the {incremental minimization problem} determined by $\Pi$ with initial data $\boldsymbol{q}^0\in \mathcal{Q}$, which reads as follows:
\begin{equation}
    \label{eqn:incremental-minimization-problem}
    \begin{split}
        \text{find}&\text{ $(\boldsymbol{q}^1,\dots,\boldsymbol{q}^N)\in \mathcal{Q}^N$ such that each $\boldsymbol{q}^i$ is a minimizer}\\
        &\hspace{5mm}\text{of $\boldsymbol{q}\mapsto \mathcal{E}(t_i,\boldsymbol{q})+\mathcal{D}(\boldsymbol{q}^{i-1},\boldsymbol{q})$ for $i=1,\dots,N$.}
    \end{split}
\end{equation}

The next result states the existence of solutions of \eqref{eqn:incremental-minimization-problem} and collects their main properties. Recall the definition of the total energy $\mathcal{E}$ and of the dissipation distance $\mathcal{D}$ in \eqref{eqn:total-energy} and \eqref{eqn:dissipation-distance}, respectively. 

\begin{proposition}[Solutions of the incremental minimization problem]
\label{prop:solutions-imp}
Assume $p>3$ and $\mathcal{Y} \neq \emptyset$. Suppose that $W$ is continuous and satisfies \eqref{eqn:W-coercivity}--\eqref{eqn:W-polyconvex}, and that the applied loads satisfy \eqref{eqn:applied-loads-regularity}. Let $\Pi=(t_0,\dots,t_N)$ be a partition of $[0,T]$ and let $\boldsymbol{q}^0 \in \mathcal{Q}$.
Then, the incremental minimization problem \eqref{eqn:incremental-minimization-problem} admits a solution $(\boldsymbol{q}^1,\dots,\boldsymbol{q}^N)\in \mathcal{Q}^N$. Moreover, if $\boldsymbol{q}^0$ is such that
\begin{equation}
    \label{eqn:stable-0}
    \mathcal{E}(0,\boldsymbol{q}^0)\leq \mathcal{E}(0,\widehat{\boldsymbol{q}})+\mathcal{D}(\boldsymbol{q}^0,\widehat{\boldsymbol{q}})
\end{equation}
for every $\widehat{ \boldsymbol{q}}\in \mathcal{Q}$, then the following holds:
\begin{equation}
    \label{eqn:imp-stability}
    \forall\,i=1,\dots,N,\:\forall\,\widehat{\boldsymbol{q}}\in \mathcal{Q},\quad \mathcal{E}(t_i,\boldsymbol{q}^i)\leq \mathcal{E}(t_i,\widehat{\boldsymbol{q}})+\mathcal{D}(\boldsymbol{q}^i,\widehat{\boldsymbol{q}}),
\end{equation}
\begin{equation}
    \label{eqn:imp-energy-inequality}
    \forall\,i=1,\dots,N,\quad \mathcal{E}(t_i,\boldsymbol{q}^i)-\mathcal{E}(t_{i-1},\boldsymbol{q}^{i-1})+\mathcal{D}(\boldsymbol{q}^{i-1},\boldsymbol{q}^i)\leq \int_{t_{i-1}}^{t_i} \partial_t \mathcal{E}(\tau,\boldsymbol{q}^{i-1})\,\d\tau,\vspace{-3mm}
\end{equation}
\begin{equation}
    \label{eqn:imp-energy-estimate}
    \forall\,i=1,\dots,N,\quad \mathcal{E}(t_i,\boldsymbol{q}^i)+M+\sum_{j=1}^i \mathcal{D}(\boldsymbol{q}^{j-1},\boldsymbol{q}^j)\leq (\mathcal{E}(0,\boldsymbol{q}^0)+M)\,e^{L t_i}.
\end{equation}
\end{proposition}
\begin{proof}
The main point is to prove the existence of solutions of \eqref{eqn:incremental-minimization-problem}. Given a solution of \eqref{eqn:incremental-minimization-problem} where $\boldsymbol{q}^0$ satisfies \eqref{eqn:stable-0}, then \eqref{eqn:imp-stability}-\eqref{eqn:imp-energy-estimate} are obtained by standard computations as in \cite[Theorem 3.2]{mielke}.

It is sufficient to show that, for $\tilde{t}\in [0,T]$ and $\widetilde{\boldsymbol{q}}\in \mathcal{Q}$ fixed, the auxiliary functional $\mathcal{F}\colon \mathcal{Q} \to \R$ given by
    $\mathcal{F}(\boldsymbol{q})\coloneqq \mathcal{E}(\tilde{t},\boldsymbol{q})+\mathcal{D}(\widetilde{\boldsymbol{q}},\boldsymbol{q})$,
admits a minimizer in $\mathcal{Q}$.
As $\mathcal{D}$ is positive, from \eqref{eqn:total-energy-coercivity} we have {\MMM
\begin{equation}
    \label{eqn:F-coercivity}
    \mathcal{F}(\boldsymbol{q})\geq  C_0 ||\nabla \boldsymbol{y}||_{L^p(\Omega;\rtt)}^p+||\gamma(\det \nabla \boldsymbol{y})||_{L^1(\Omega)}+\alpha||\nabla \boldsymbol{m}||^2_{L^2(\Omega^{\boldsymbol{y}};\rtt)}- C_1
\end{equation}
for every $\boldsymbol{q}\in \mathcal{Q}$ with $\boldsymbol{q}=(\boldsymbol{y},\boldsymbol{m})$.}
Let $(\boldsymbol{q}_n)\subset \mathcal{Q}$ with $\boldsymbol{q}_n=(\boldsymbol{y}_n,\boldsymbol{m}_n)$ be a minimizing sequence for $\mathcal{F}$, namely such that $\mathcal{F}(\boldsymbol{q}_n)\to \inf_{\mathcal{Q}}\mathcal{F}$. In particular, $\sup_{n\in \N}\mathcal{F}(\boldsymbol{q}_n)<+\infty$, so that \eqref{eqn:F-coercivity} yields \eqref{eqn:compactness-bound}. By Proposition \ref{prop:compactness}, there exists $\boldsymbol{q}\in \mathcal{Q}$ such that, up to subsequences, we have $\boldsymbol{q}_n \to \boldsymbol{q}$ in $\mathcal{Q}$ and $\boldsymbol{m}_n\circ \boldsymbol{y}_n \to \boldsymbol{m}\circ \boldsymbol{y}$ in $L^{\RRR a}(\Omega;\R^3)$ for every $1\leq {\RRR a}<\infty$. Arguing as in the proof of Theorem \ref{thm:existence-minimizers}, we prove \eqref{eqn:existence-lsc}
while, exploiting the weak continuity of the trace operator, we get
\begin{equation}
    \label{eqn:auxilliary-L}
    \mathcal{L}(\tilde{t},\boldsymbol{q})=\lim_{n\to \infty} \mathcal{L}(\tilde{t},\boldsymbol{q}_n).
\end{equation}
By the weak continuity of {\MMM Jacobian} minors, $\cof\,\nabla \boldsymbol{y}_n \wk \cof\,\nabla\boldsymbol{y}$ in $L^{{p}/{2}}(\Omega;\rtt)$. {\MMM This, combined with the convergence of  $(\boldsymbol{m}_n \circ \boldsymbol{y}_n)$ in $L^{({p}/{2})'}(\Omega;\R^3)$, yields $\mathcal{Z}(\boldsymbol{q}_n)\wk \mathcal{Z}(\boldsymbol{q})$ in $L^1(\Omega;\R^3)$} and, by the lower semicontinuity of the norm, we deduce
\begin{equation}
    \label{eqn:auxilliary-D}
    \mathcal{D}(\widetilde{\boldsymbol{q}},\boldsymbol{q})\leq \liminf_{n\to \infty} \mathcal{D}(\widetilde{\boldsymbol{q}},\boldsymbol{q}_n).
\end{equation}
Finally, combining \eqref{eqn:existence-lsc} and \eqref{eqn:auxilliary-L}--\eqref{eqn:auxilliary-D}, we obtain
\begin{equation*}
    \mathcal{F}(\boldsymbol{q})\leq \liminf_{n\to \infty} \mathcal{F}(\boldsymbol{q}_n),  
\end{equation*}
so that $\boldsymbol{q}$ is a minimizer of $\mathcal{F}$.
\end{proof}

{\MMM Unfortunately, in our setting, we can not proceed with the second step of the proof of the existence of energetic solutions. This is due to a lack of compactness in the dissipative variable which is typical of large-strain theories. Therefore, in the next subsection, we propose a regularization of the model in the spirit of gradient polyconvexity \cite{benesova.kruzik.schloemerkemper}.}

\subsection{Regularized setting}
\label{subs:time-continuous} 
Henceforth, we regularize the problem as follows. Recalling \eqref{eqn:admissible-deformations},  we restrict ourselves to the class of deformations
\begin{equation}
        \label{eqn:admissible-deformation-regularized}
        \widetilde{\mathcal{Y}}\coloneqq \left \{\boldsymbol{y}\in \mathcal{Y}:\:\cof\,\nabla \boldsymbol{y}\in BV(\Omega;\R^{3 \times 3}) \right \},
\end{equation}
so that the corresponding class of admissible states is given by
\begin{equation}
    \widetilde{\mathcal{Q}}\coloneqq \left \{(\boldsymbol{y},\boldsymbol{m}):\:\boldsymbol{y}\in \widetilde{\mathcal{Y}},\:\boldsymbol{m}\in W^{1,2}(\Omega^{\boldsymbol{y}};\S^2) \right \}.
\end{equation}
{\MMM Equivalently,} in \eqref{eqn:admissible-deformation-regularized}, {\MMM we require that the distributional gradient of $\cof\,\nabla \boldsymbol{y}$ is given by a bounded tensor-valued Radon measure $D(\cof\,\nabla \boldsymbol{y})\in \mathcal{M}_{\rm b}(\Omega;\R^{3 \times 3 \times 3})$.}

\begin{example}
Let $\Omega$, $P$ and $\boldsymbol{y}$ be as in Example \ref{ex:ball} and recall Example \ref{ex:ball2}. Then, $\boldsymbol{y}\in \widetilde{\mathcal{Y}}$. To see this, for every $\boldsymbol{x}\in \Omega \setminus P$ with $\boldsymbol{x}=(x_1,x_2,x_3)$, we compute
\begin{equation*}
    \cof\,\nabla \boldsymbol{y}(\boldsymbol{x})\coloneqq \begin{pmatrix}
      |x_1| & 0 & -{x_1\,x_3}/{|x_1|}\\
      0 & |x_1| & 0\\
      0 & 0 & 1
    \end{pmatrix}.
\end{equation*}
Set $u(\boldsymbol{x})\coloneqq |x_1|$ and $v(\boldsymbol{x})\coloneqq -{x_1\,x_3}/{|x_1|}$. Then $u \in W^{1,\infty}(\Omega)$, while $v\in BV(\Omega)$ since $$D_1 v=w\,\haus \mres{\{0\}\times (-1,1)^2},$$ where $D_1$ denotes the  distributional derivative with respect  to the first variable and we set $w(\boldsymbol{x})\coloneqq 2 x_3$. Therefore $\boldsymbol{y}\in \widetilde{\mathcal{Y}}$.
\end{example}

{\MMM
\begin{example}
Define $f\colon [0,1]\to \R$ by setting $f(x)\coloneqq x^2\cos^2(\pi/x^2)$ for every $0<x\leq 1$ and $f(0)\coloneqq 0$, and let $g\colon [0,1]\to \R$ be given by $g(x)\coloneqq \int_0^x f(z)\,\d z$. We have $f\in C^0([0,1])\setminus BV([0,1])$ and $g \in C^1([0,1])$. Moreover, $g$ is strictly increasing and, in turn, injective. Let $\Omega\coloneqq (0,1)^3$ and define $\boldsymbol{y}\colon \Omega \to \R^3$ by $\boldsymbol{y}(\boldsymbol{x})\coloneqq (x_1,x_2,g(x_1)x_3)$, where $\boldsymbol{x}=(x_1,x_2,x_3)$. In this case, $\boldsymbol{y}\in C^1(\closure{\Omega};\R^3)$ is a homeomorphism and $\det\nabla\boldsymbol{y}>0$. However
\begin{equation*}
    \cof\,\nabla \boldsymbol{y}(\boldsymbol{x})\coloneqq \begin{pmatrix}
      g(x_1) & 0 & -f(x_1)x_3\\
      0 & g(x_1) & 0\\
      0 & 0 & 1
    \end{pmatrix},
\end{equation*}
so that $\cof\nabla\boldsymbol{y}\notin BV(\Omega;\rtt)$. In particular, for $\Gamma\coloneqq \{1\}\times (0,1)^2$ and $\overline{\boldsymbol{y}}\coloneqq\boldsymbol{id}$, there holds $\boldsymbol{y}\in\mathcal{Y}\setminus\widetilde{\mathcal{Y}}$.
\end{example}
}

Recalling \eqref{eqn:energy-E}, the regularized magnetoelastic energy $\widetilde{E}\colon \widetilde{\mathcal{Q}}\to \R$ is given by
\begin{equation}
\label{eqn:E-reg}
    \widetilde{E}(\boldsymbol{q})\coloneqq E(\boldsymbol{q})+|D (\cof\,\nabla \boldsymbol{y})|(\Omega), 
\end{equation}
where $\boldsymbol{q}=(\boldsymbol{y},\boldsymbol{m})$ and $|D(\cof\,\nabla \boldsymbol{y})|(\Omega)$ denotes the total variation of  the measure $D(\cof\,\nabla \boldsymbol{y})$ over  $\Omega$.
The corresponding total energy $\widetilde{\mathcal{E}}\colon [0,T] \times \widetilde{\mathcal{Q}}\to \R$ is defined as
\begin{equation}
    \widetilde{\mathcal{E}}(t,\boldsymbol{q})\coloneqq \widetilde{E}(\boldsymbol{q})-\mathcal{L}(t,\boldsymbol{q}),
\end{equation}
where $\mathcal{L}$ is given by \eqref{eqn:applied-loads-functional}. Also, analogously to \eqref{eqn:total-energy-gronwall}, there holds
\begin{equation}
    \label{eqn:total-energy-regularized-gronwall}
    \widetilde{\mathcal{E}}(t,\boldsymbol{q})+M \leq (\widetilde{\mathcal{E}}(s,\boldsymbol{q})+M)e^{L(t-s)}
\end{equation}
for every $\boldsymbol{q}\in \widetilde{\mathcal{Q}}$ and $s,t\in [0,T]$ with $s<t$.

{\MMM The second main result of the paper states the existence of energetic solutions for the regularized model.}

\begin{theorem}[Existence of energetic solutions]
\label{thm:existence-energetic-solution}
Assume $p>3$ and $\widetilde{\mathcal{Y}}\neq \emptyset$. Suppose that $W$ is continuous and satisfies \eqref{eqn:W-coercivity}--\eqref{eqn:W-polyconvex}, and that the applied loads satisfy \eqref{eqn:applied-loads-regularity}. Then, for every $\boldsymbol{q}^0 \in \widetilde{\mathcal{Q}}$ satisfying 
\begin{equation}
    \label{eqn:stable-0-regularization}
    \forall\,\widehat{\boldsymbol{q}}\in \widetilde{\mathcal{Q}},\quad \widetilde{\mathcal{E}}(0,\boldsymbol{q}^0)\leq \widetilde{\mathcal{E}}(0,\widehat{\boldsymbol{q}})+\mathcal{D}(\boldsymbol{q}^0,\widehat{\boldsymbol{q}}),
\end{equation}
there exists an energetic solution $\boldsymbol{q}\colon [0,T]\to \widetilde{\mathcal{Q}}$ of the regularized model which fulfills the initial condition $\boldsymbol{q}(0)=\boldsymbol{q}^0$. Namely, the following  \emph{stability condition} and \emph{energy balance} hold:
\begin{equation}
    \label{eqn:energetic-solution-stability}
    \forall t \in [0,T],\:\:\forall\, \widehat{\boldsymbol{q}}\in \widehat{\mathcal{Q}}, \quad \widetilde{\mathcal{E}}(t,\boldsymbol{q}(t))\leq \widetilde{\mathcal{E}}(t,\widehat{\boldsymbol{q}})+\mathcal{D}(\boldsymbol{q}(t),\widehat{\boldsymbol{q}}),
\end{equation}
\begin{equation}
    \label{eqn:energetic-solution-energy-balance}
    \forall t \in [0,T], \quad \widetilde{\mathcal{E}}(t,\boldsymbol{q}(t))+\var_{\mathcal{D}}(\boldsymbol{q};[0,t])=\widetilde{\mathcal{E}}(0,\boldsymbol{q}^0)+\int_0^t \partial_t \widetilde{\mathcal{E}}(\tau,\boldsymbol{q}(\tau))\,\d\tau.
\end{equation}
\end{theorem}

{\MMM Given the highly nonconvex character of the energy, we can not establish any regularity of the solution. However, by applying a suitable version of the Measurable Selection Lemma, the existence of a measurable energetic solution can be ensured.}

{\MMM As already mentioned, the proof Theorem \ref{thm:existence-energetic-solution} proceeds by time-discretization.}
Let $\Pi=(t_0,\dots,t_N)$ be a partition of $[0,T]$. The {\MMM {regularized}} {incremental minimization problem} determined by $\Pi$ with initial data $\boldsymbol{q}^0\in \mathcal{Q}$ reads as follows:
\begin{equation}
    \label{eqn:incremental-minimization-problem-regularization}
    \begin{split}
        \text{find}&\text{ $(\boldsymbol{q}^1,\dots,\boldsymbol{q}^N)\in \widetilde{\mathcal{Q}}^N$ such that each $\boldsymbol{q}^i$ is a minimizer}\\
        &\hspace{6mm}\text{of $\boldsymbol{q}\mapsto \widetilde{\mathcal{E}}(t_i,\boldsymbol{q})+\mathcal{D}(\boldsymbol{q}^{i-1},\boldsymbol{q})$ for $i=1,\dots,N$.}
    \end{split}
\end{equation}

{\MMM The next result provides the analogous of Proposition \ref{prop:solutions-imp} in the regularized setting.}

\begin{proposition}[Solutions of the {\MMM regularized} incremental minimization problem]
\label{prop:solutions-imp-regularization}
Assume $p>3$ and $\widetilde{\mathcal{Y}} \neq \emptyset$. Suppose that $W$ is continuous and satisfies \eqref{eqn:W-coercivity}--\eqref{eqn:W-polyconvex}, and that the applied loads satisfy \eqref{eqn:applied-loads-regularity}. Let $\Pi=(t_0,\dots,t_N)$ be a partition of $[0,T]$ and let $\boldsymbol{q}^0 \in \widetilde{\mathcal{Q}}$.
Then, the incremental minimization problem \eqref{eqn:incremental-minimization-problem-regularization} admits a solution $(\boldsymbol{q}^1,\dots,\boldsymbol{q}^N)\in \widetilde{\mathcal{Q}}^N$. Moreover, if $\boldsymbol{q}^0$ satisfies
\eqref{eqn:stable-0-regularization},
then the following holds:
\begin{equation}
    \label{eqn:imp-stability-regularization}
    \forall\,i=1,\dots,N,\:\:\forall\,\widehat{\boldsymbol{q}}\in \widetilde{\mathcal{Q}},\quad \widetilde{\mathcal{E}}(t_i,\boldsymbol{q}^i)\leq \widetilde{\mathcal{E}}(t_i,\widehat{\boldsymbol{q}})+\mathcal{D}(\boldsymbol{q}^i,\widehat{\boldsymbol{q}}),
\end{equation}
\begin{equation}
    \label{eqn:imp-energy-inequality-regularization}
    \forall\,i=1,\dots,N,\quad \widetilde{\mathcal{E}}(t_i,\boldsymbol{q}^i)-\widetilde{\mathcal{E}}(t_{i-1},\boldsymbol{q}^{i-1})+\mathcal{D}(\boldsymbol{q}^{i-1},\boldsymbol{q}^i)\leq \int_{t_{i-1}}^{t_i} \partial_t \widetilde{\mathcal{E}}(\tau,\boldsymbol{q}^{i-1})\,\d\tau, \vspace{-3mm}
\end{equation} 
\begin{equation}
    \label{eqn:imp-a-priori-estimate-regularization}
    \forall\,i=1,\dots,N,\quad \widetilde{\mathcal{E}}(t_i,\boldsymbol{q}^i)+M+\sum_{j=1}^i \mathcal{D}(\boldsymbol{q}^{j-1},\boldsymbol{q}^j)\leq (\widetilde{\mathcal{E}}(0,\boldsymbol{q}^0)+M)\,e^{L t_i}.
\end{equation}
\end{proposition}
\begin{proof}
Again, the main point is to prove the existence of solutions to \eqref{eqn:incremental-minimization-problem-regularization}. Hence, we show that the auxiliary functional $\widetilde{\mathcal{F}}\colon \widetilde{\mathcal{Q}}\to \R$ defined by $\widetilde{\mathcal{F}}(\boldsymbol{q})\coloneqq \widetilde{\mathcal{E}}(\tilde{t},\boldsymbol{q})+\mathcal{D}(\widetilde{\boldsymbol{q}},\boldsymbol{q})$, where $\tilde{t}\in[0,T]$ and $\widetilde{\boldsymbol{q}}\in \widetilde{\mathcal{Q}}$ are fixed, admits a minimizer in $\widetilde{\mathcal{Q}}$. {\MMM The proof goes as the one of Proposition \ref{prop:solutions-imp}. Let $(\boldsymbol{q}_n)\subset\widetilde{\mathcal{Q}}$ with $\boldsymbol{q}_n=(\boldsymbol{y}_n,\boldsymbol{m}_n)$ be a minimizing sequence for $\widetilde{\mathcal{F}}$, namely such that $\widetilde{\mathcal{F}}(\boldsymbol{q}_n)\to\inf_{\widetilde{\mathcal{Q}}}\widetilde{\mathcal{F}}$. We have $\sup_{n\in\N}\widetilde{\mathcal{F}}(\boldsymbol{q}_n)<+\infty$. From this, exploiting the coercivity in \eqref{eqn:total-energy-coercivity} and applying Proposition \ref{prop:compactness}, we find $\boldsymbol{q}=(\boldsymbol{y},\boldsymbol{m})\in\mathcal{Q}$ such that, up to subsequences, the convergences in \eqref{eqn:Q-convergence-m}--\eqref{eqn:Q-convergence-strong} and \eqref{eqn:convergence-composition} hold true. These allow us to establish \eqref{eqn:existence-lsc} and \eqref{eqn:auxilliary-L}--\eqref{eqn:auxilliary-D} as in Proposition \ref{prop:solutions-imp}.
We also deduce that the sequence $(\cof\nabla\boldsymbol{y}_n)$ is bounded in $BV(\Omega;\rtt)$. Thus, up to subsequences, there hold
\begin{equation}
\label{eqn:auxilliary-BV}
    \text{$\cof\,\nabla \boldsymbol{y}_n\to\boldsymbol{G}$ in $L^{3/2}(\Omega;\rtt)$, \qquad $D(\cof\,\nabla \boldsymbol{y}_n)\wks D\boldsymbol{G}$ in $\mathcal{M}_{\rm b}(\Omega;\R^{3\times 3\times 3})$,}
\end{equation}
for some $\boldsymbol{G}\in BV(\Omega;\rtt)$. By \eqref{eqn:Q-convergence-y} and the weak continuity of Jacobian minors, $\boldsymbol{G}=\cof\nabla\boldsymbol{y}$ and, in particular, $\boldsymbol{y}\in\widetilde{\mathcal{Y}}$ and $\boldsymbol{q}\in\widetilde{\mathcal{Q}}$. Moreover, by the lower semicontinuity of the total variation, we obtain
\begin{equation*}
    |D(\cof\nabla\boldsymbol{y})|(\Omega)\leq\liminf_{n\to\infty}|D(\cof\nabla\boldsymbol{y}_n)|(\Omega).
\end{equation*}
This, combined with \eqref{eqn:existence-lsc} and \eqref{eqn:auxilliary-L}--\eqref{eqn:auxilliary-D}, yields
\begin{equation*}
    \widetilde{\mathcal{F}}(\boldsymbol{q})\leq\liminf_{n\to\infty}\widetilde{\mathcal{F}}(\boldsymbol{q}_n)
\end{equation*}
and, in turn, $\boldsymbol{q}$ is a minimizer of $\widetilde{\mathcal{F}}$.}
\end{proof}

{\MMM 
\begin{remark}[Gradient polyconvexity]
The regularization introduced in Subsection \ref{subs:time-continuous} makes assumption \eqref{eqn:W-polyconvex} superfluous in the proof of Proposition \ref{prop:solutions-imp-regularization} as well as in the rest of our analysis. By \eqref{eqn:auxilliary-BV},  up to subsequences, we have $\cof\nabla \boldsymbol{y}_n\to\cof\nabla\boldsymbol{y}$  almost everywhere. As  observed in \cite{benesova.kruzik.schloemerkemper}, this entails
$\nabla \boldsymbol{y}_n \to \nabla \boldsymbol{y}$ almost everywhere.  Indeed, exploiting the identity $\det(\cof\boldsymbol{F})=(\det\boldsymbol{F})^2$ for every $\boldsymbol{F}\in\rtt$, we see that $\det \nabla \boldsymbol{y}_n\to\det\nabla\boldsymbol{y}$ almost everywhere. Then, by the formula $\boldsymbol{F}^{-1}=(\det\boldsymbol{F})^{-1}(\adj\boldsymbol{F})$ for every $\boldsymbol{F}\in\rtt_+$, we obtain $(\nabla\boldsymbol{y}_n)^{-1}\to(\nabla\boldsymbol{y})^{-1}$ almost everywhere and the claim follows by the continuity of the map $\boldsymbol{F}\mapsto \boldsymbol{F}^{-1}$ on $\rtt_+$. The almost everywhere convergence of $(\nabla\boldsymbol{y}_n)$ and $(\boldsymbol{m}_n\circ\boldsymbol{y}_n)$, which can be assumed by \eqref{eqn:convergence-composition}, entails the lower semicontinuity of the elastic energy in \eqref{eqn:existence-lsc-el}  by a simple application of the Fatou Lemma. Therefore,  assumption \eqref{eqn:W-polyconvex} is actually not necessary in the regularized setting. 
\end{remark}
}

{\MMM In the next proposition, we consider the piecewise-constant interpolants corresponding to solutions of \eqref{eqn:incremental-minimization-problem-regularization} and we collect their main properties. }
Recall the definition of variation with respect to $\mathcal{D}$ in \eqref{eqn:dissipation-variation}.

\begin{proposition}[Piecewise-constant interpolants]
\label{prop:piecewise-constant-interpolants}
Assume $p>3$ and $\widetilde{\mathcal{Y}}\neq \emptyset$. Suppose that $W$ is continuous and satisfies \eqref{eqn:W-coercivity}--\eqref{eqn:W-polyconvex}, and that the applied loads satisfy \eqref{eqn:applied-loads-regularity}. Let $\Pi=(t_0,\dots,t_N)$ be a partition of $[0,T]$ and let $\boldsymbol{q}^0 \in \widetilde{\mathcal{Q}}$ satisfy \eqref{eqn:stable-0-regularization}. Let $(\boldsymbol{q}^1,\dots,\boldsymbol{q}^N)\in \widetilde{\mathcal{Q}}^N$ be a solution of the regularized incremental minimization problem \eqref{eqn:incremental-minimization-problem-regularization} and define the (right-continuous) piecewise-constant interpolant $\boldsymbol{q}_\Pi \colon [0,T]\to \widetilde{\mathcal{Q}}$ as
\begin{equation}
    \label{eqn:pc-interpolant}
    \boldsymbol{q}_\Pi (t)\coloneqq
    \begin{cases}
    \boldsymbol{q}^{i-1} & \text{if $t \in [t_{i-1},t_i)$ for some $i=1,\dots,N$,}\\
    \boldsymbol{q}^N & \text{if $t=T$.}
    \end{cases}
\end{equation}
Then, the following holds:
\begin{equation}
    \label{eqn:pc-interpolants-stability}
    \forall t \in \Pi,\:\:\forall\, \widehat{\boldsymbol{q}}\in \widetilde{\mathcal{Q}},\quad \widetilde{\mathcal{E}}(t,\boldsymbol{q}_\Pi(t))\leq \widetilde{\mathcal{E}}(t,\widehat{\boldsymbol{q}})+\mathcal{D}(\boldsymbol{q}_\Pi(t),\widehat{\boldsymbol{q}}),
\end{equation}
\begin{equation}
    \label{eqn:pc-interpolants-energy-inequality}
    \forall s,t\in \Pi:\:s<t,\quad \widetilde{\mathcal{E}}(t,\boldsymbol{q}_\Pi(t))-\widetilde{\mathcal{E}}(s,\boldsymbol{q}_\Pi(s))+\var_{\mathcal{D}}(\boldsymbol{q}_\Pi;[s,t])\leq \int_s^t \partial_t \widetilde{\mathcal{E}}(\tau,\boldsymbol{q}_\Pi(\tau))\,\d\tau, 
\end{equation}
\begin{equation}
    \label{eqn:pc-interpolants-a-priori-estimate}
    \forall t \in [0,T],\quad \widetilde{\mathcal{E}}(t,\boldsymbol{q}_\Pi(t))+M+\var_{\mathcal{D}}(\boldsymbol{q}_\Pi;[0,t])\leq (\widetilde{\mathcal{E}}(0,\boldsymbol{q}^0)+M) e^{L t}.
\end{equation}
\end{proposition}
\begin{proof}
Claims \eqref{eqn:pc-interpolants-stability}--\eqref{eqn:pc-interpolants-energy-inequality} follow immediately from \eqref{eqn:imp-stability-regularization}--\eqref{eqn:imp-energy-inequality-regularization}, respectively. We prove \eqref{eqn:pc-interpolants-a-priori-estimate}. Let $t \in [0,T]$ and let $i \in \{1,\dots,N\}$ be such that $t_{i-1}\leq t<t_i$. In this case, we have
\begin{equation*}
    \boldsymbol{q}_\Pi(t)=\boldsymbol{q}^{i-1},\qquad \var_{\mathcal{D}}(\boldsymbol{q}_\Pi;[0,t])=\sum_{j=1}^{i-1} \mathcal{D}(\boldsymbol{q}^{j-1},\boldsymbol{q}^j).
\end{equation*}
Thus, using \eqref{eqn:total-energy-regularized-gronwall} and \eqref{eqn:imp-a-priori-estimate-regularization}, we compute
\begin{equation}
    \begin{split}
    \widetilde{\mathcal{E}}(t,\boldsymbol{q}_{\Pi}(t))+M +\var_{\mathcal{D}}(\boldsymbol{q}_\Pi;[0,t])&\leq \left( \widetilde{\mathcal{E}}(t_{i-1},\boldsymbol{q}^{i-1})+M  \right)e^{L(t-t_{i-1})}+\sum_{j=1}^{i-1} \mathcal{D}(\boldsymbol{q}^{j-1},\boldsymbol{q}^j)\\
    &\leq \left (\widetilde{\mathcal{E}}(t_{i-1},\boldsymbol{q}_{i-1})+M+\sum_{j=1}^{i-1} \mathcal{D}(\boldsymbol{q}^{j-1},\boldsymbol{q}^j)    \right ) e^{L(t-t_{i-1})}\\
    &\leq \left ( \widetilde{\mathcal{E}}(0,\boldsymbol{q}^0)+M  \right) e^{L t}.
    \end{split}
\end{equation}
\end{proof}

In the proof of Theorem \ref{thm:existence-energetic-solution}, we will use the following version of the Helly Selection Principle, {\MMM which is a special case of \cite[Theorem 3.2]{mainik.mielke05}}.

\begin{lemma}[Helly Selection Principle]
\label{lem:helly}
Let $Z$ be a Banach space and let  $\mathcal{K}\subset Z$ be compact.
Let 
$(\boldsymbol{z}_n)\subset BV([0,T];Z)$ be  such that for every $n \in \N$ there holds
\begin{equation}
    \label{eqn:helly-compact}
    \forall t \in [0,T],\:\:\boldsymbol{z}_n(t)\in \mathcal{K}
\end{equation}
and
\begin{equation}
    \label{eqn:helly-variation}
    \sup_{n \in \N}\var(\boldsymbol{z}_n;[0,T])<+\infty. 
\end{equation}
Then, there exist a subsequence $(\boldsymbol{z}_{n_k})$ and a map $\boldsymbol{z}\in BV([0,T];Z)$ such that there holds:
\begin{equation}
    \label{eqn:helly-2}
    \forall t \in [0,T],\quad \text{$\boldsymbol{z}_{n_k}(t)\to\boldsymbol{z}(t)$ in $Z$}.
\end{equation}
\end{lemma}

The proof of Theorem \ref{thm:existence-energetic-solution} follows rigorously the well-established scheme introduced in \cite{francfort.mielke}. Therefore, we simply show how to lead the argument back to the original scheme. For additional details we refer to \cite[Theorem 5.2]{mielke}.

\begin{proof}[Proof of Theorem \ref{thm:existence-energetic-solution}]
Following \cite[Theorem 5.2]{mielke}, we subdivide the proof into five steps.

\textbf{Step 1 (A priori estimates).} Let $(\Pi_n)$ be a sequence of partitions of $[0,T]$ with $\Pi_n=(t_0^n,\dots,t_{N_n}^n)$ such that $|\Pi_n|\coloneqq \max \{ t_i^n-t_{i-1}^n: \: i=1, \dots,N_n \}\to 0$, as $n \to \infty$. For every $n \in \N$, by Proposition \ref{prop:solutions-imp-regularization}, the incremental minimization problem \eqref{eqn:incremental-minimization-problem-regularization} determined by $\Pi_n$ admits a solution and, by Proposition \ref{prop:piecewise-constant-interpolants}, the corresponding piecewise-constant interpolant $\boldsymbol{q}_n \coloneqq \boldsymbol{q}_{\Pi_n}$ with $\boldsymbol{q}_n=(\boldsymbol{y}_n,\boldsymbol{m}_n)$ defined according to \eqref{eqn:pc-interpolant} satisfies the following:
\begin{equation}
    \label{eqn:pc-interpolants-stability-n}
    \forall t \in \Pi_n,\:\:\forall\, \widehat{\boldsymbol{q}}\in \widetilde{\mathcal{Q}},\quad \widetilde{\mathcal{E}}(t,\boldsymbol{q}_n(t))\leq \widetilde{\mathcal{E}}(t,\widehat{\boldsymbol{q}})+\mathcal{D}(\boldsymbol{q}_n(t),\widehat{\boldsymbol{q}}),
\end{equation}
\begin{equation}
    \label{eqn:pc-interpolants-energy-inequality-n}
    \forall s,t\in \Pi_n:\:s<t,\quad \widetilde{\mathcal{E}}(t,\boldsymbol{q}_n(t))-\widetilde{\mathcal{E}}(s,\boldsymbol{q}_n(s))+\var_{\mathcal{D}}(\boldsymbol{q}_n;[s,t])\leq \int_s^t \partial_t \widetilde{\mathcal{E}}(\tau,\boldsymbol{q}_n(\tau))\,\d\tau,
\end{equation}
\begin{equation}
    \label{eqn:pc-interpolants-a-priori-estimate-n}
    \forall t \in [0,T],\quad \widetilde{\mathcal{E}}(t,\boldsymbol{q}_n(t))+M+\var_{\mathcal{D}}(\boldsymbol{q}_n;[0,t])\leq (\widetilde{\mathcal{E}}(0,\boldsymbol{q}^0)+M) e^{L t}.
\end{equation}
In particular, from \eqref{eqn:pc-interpolants-a-priori-estimate-n}, we deduce that, for every $n \in \N$, there hold
\begin{equation}
    \label{eqn:a-priori-bound}
    \sup_{n \in \N} \left \{\sup_{t \in [0,T]}\widetilde{\mathcal{E}}(t,\boldsymbol{q}_n(t))+\var_{\mathcal{D}}(\boldsymbol{q}_n;[0,T]) \right \} \leq C
\end{equation}
for some constant $C(\boldsymbol{q}^0,M,L,T)>0$.

\textbf{Step 2 (Selection of subsequences).} From \eqref{eqn:total-energy-coercivity} and \eqref{eqn:a-priori-bound}, for every $n \in \N$ and $t \in [0,T]$ we have
\begin{align*}
    &\hspace*{2mm}||\nabla \boldsymbol{y}_n(t)||_{L^p(\Omega;\rtt)}\leq C,\qquad \hspace{7mm} ||\nabla \boldsymbol{m}_n(t)||_{L^2(\Omega^{\boldsymbol{y}_n(t)};\rtt)}\leq C,\\
    &||\gamma(\det \nabla \boldsymbol{y}_n(t))||_{L^1(\Omega)}\leq C, \qquad ||D(\cof\,\nabla \boldsymbol{y}_n(t))||_{\mathcal{M}_{\rm b}(\Omega;\R^{3 \times 3 \times 3})}\leq C.
\end{align*}
This shows that each maps of the sequence $(\boldsymbol{q}_n)$ takes values in the set $\widetilde{\mathcal{K}}\subset \widetilde{\mathcal{Q}}$ defined as
\begin{equation*}
    \begin{split}
    \widetilde{\mathcal{K}}\coloneqq \bigg \{ \widehat{\boldsymbol{q}}=(\widehat{\boldsymbol{y}},\widehat{\boldsymbol{m}})\in \widetilde{\mathcal{Q}}:\quad &||\nabla \widehat{\boldsymbol{y}}||_{L^p(\Omega;\rtt)}\leq C,\quad \hspace*{3mm}||\nabla \widehat{\boldsymbol{m}}||_{L^2(\Omega^{\widehat{\boldsymbol{y}}};\rtt)}\leq C,\\
    &||\gamma(\det \nabla \widehat{\boldsymbol{y}})||_{L^1(\Omega)}\leq C, \quad ||D(\cof\,\nabla \widehat{\boldsymbol{y}})||_{\mathcal{M}_{\rm b}(\Omega;\R^{3 \times 3 \times 3})}\leq C   \bigg \}.
    \end{split}
\end{equation*}
Applying Proposition \ref{prop:compactness} and arguing as in the proof of Proposition \ref{prop:solutions-imp-regularization}, we prove the following:
\begin{equation}
    \label{eqn:compactness-property}
    \begin{split}
    &\text{for every $(\widehat{\boldsymbol{q}}_n)\subset \widetilde{\mathcal{K}}$ with $\widehat{\boldsymbol{q}}_n=(\widehat{\boldsymbol{y}}_n,\widehat{\boldsymbol{m}}_n)$ there exist  $(\widehat{\boldsymbol{q}}_{n_k})$ and $\widehat{\boldsymbol{q}}\in \widetilde{\mathcal{Q}}$ with $\widehat{\boldsymbol{q}}=(\widehat{\boldsymbol{y}},\widehat{\boldsymbol{m}})$}\\
    &\hspace*{2mm}\text{such that $\widehat{\boldsymbol{q}}_{n_k}\to \widehat{\boldsymbol{q}}$ in $\widetilde{\mathcal{Q}}$, \hspace{2mm} $\widehat{\boldsymbol{m}}_{n_k}\circ \widehat{\boldsymbol{y}}_{n_k} \to \widehat{\boldsymbol{m}}\circ \widehat{\boldsymbol{y}}$ in $L^{\RRR a}(\Omega;\R^3)$ for every $1 \leq {\RRR a} < \infty$,}\\
    &\hspace{2mm}\text{{\MMM $\mathcal{Z}(\widehat{\boldsymbol{q}}_{n_k})\to\mathcal{Z}(\widehat{\boldsymbol{q}})$ in $L^1(\Omega;\R^3)$} \hspace{1mm}and\hspace{1mm} $D(\cof\,\nabla \widehat{\boldsymbol{y}}_{n_k})\wks D(\cof\,\nabla \widehat{\boldsymbol{y}})$ in $\mathcal{M}_{\rm b}(\Omega;\R^{3 \times 3 \times 3})$.}
    \end{split}
\end{equation}
In particular, {\MMM note the strong convergence of $(\mathcal{Z}(\widehat{\boldsymbol{q}}_{n_k}))$ in $L^1(\Omega;\R^3)$ which is deduced from the strong convergence of $(\cof\,\nabla\widehat{\boldsymbol{y}}_{n_k})$ in $L^{3/2}(\Omega;\R^3)$ (see \eqref{eqn:auxilliary-BV}) and the  strong convergence of $(\boldsymbol{m}_{n_k}\circ \boldsymbol{y}_{n_k})$ in $L^3(\Omega;\R^3)$. Thus,}
the set 
\begin{equation*}
    \mathcal{K}\coloneqq \left \{\mathcal{Z}(\widehat{\boldsymbol{q}}):\:\:  \widehat{\boldsymbol{q}}\in \widetilde{\mathcal{K}} \right \}
\end{equation*}
is compact with respect to the strong topology of $L^1(\Omega;\R^3)$. 
Now, consider the sequence $(\boldsymbol{z}_n)\subset BV([0,T];L^1(\Omega;\R^3))$ with $\boldsymbol{z}_n(t)\coloneqq \mathcal{Z}(\boldsymbol{q}_n(t))$ for every $t \in [0,T]$.
Setting $Z=L^1(\Omega;\R^3)$, the sequence $(\boldsymbol{z}_n)$ satisfies \eqref{eqn:helly-compact} by construction, as the the maps of the sequence $(\boldsymbol{q}_n)$ take values in $\widetilde{\mathcal{K}}$, while  \eqref{eqn:helly-variation} holds in view of \eqref{eqn:a-priori-bound}. Therefore, by Lemma \ref{lem:helly}, there exist a subsequence $(\boldsymbol{z}_{n_k})$ and a map $\boldsymbol{z}\in BV([0,T];L^1(\Omega;\R^3))$ such that \eqref{eqn:helly-2} holds.

For every $n \in \N$, define $\vartheta_n \colon [0,T] \to \R$ by setting $\vartheta_n(t)\coloneqq \partial_t \widetilde{\mathcal{E}}(t,\boldsymbol{q}_n(t))$. By \eqref{eqn:time-derivative-total-energy-estimate} and \eqref{eqn:a-priori-bound}, the sequence $(\vartheta_n)$ is bounded in $L^\infty(0,T)$. Hence, up to subsequences, $\vartheta_n \wks \vartheta$ in $L^\infty(0,T)$ for some $\vartheta \in L^\infty(0,T)$. If we define $\bar{\vartheta} \colon [0,T]\to \R$ as $\bar{\vartheta}(t)\coloneqq \limsup_{n \to \infty} \vartheta_n(t)$, then $\bar{\vartheta}\in L^\infty(0,T)$ and, by the Fatou Lemma, $\vartheta \leq \bar{\vartheta}$.

Finally, for every fixed $t \in [0,T]$, exploiting \eqref{eqn:compactness-property}, we select a subsequence  $(\boldsymbol{q}_{n_{k_\ell^t}}(t))$ {\MMM and some $\boldsymbol{q}(t)\in \widetilde{\mathcal{Q}}$ with $\boldsymbol{q}(t)=(\boldsymbol{y}(t),\boldsymbol{m}(t))$} such that
\begin{equation}
\label{eqn:t-dependent-subsequence}
    \begin{split}
        &\hspace{4mm}\text{$\boldsymbol{q}_{n_{k_\ell^t}}(t) \to \boldsymbol{q}(t)$ in $\widetilde{\mathcal{Q}}$}, \quad \text{$\boldsymbol{m}_{n_{k_\ell^t}} \circ \boldsymbol{y}_{n_{k_\ell^t}}(t)\to \boldsymbol{m}\circ \boldsymbol{y}(t)$ in $L^a(\Omega;\R^3)$ for every $1 \leq a<\infty$},\\
        &\text{{\MMM $\mathcal{Z}(\boldsymbol{q}_{n_{k_\ell^t}}(t))\to \mathcal{Z}(\boldsymbol{q}(t))$} in $L^1(\Omega;\rtt)$, \: $D(\cof\, \nabla \boldsymbol{y}_{n_{k_\ell^t}}(t))\wks D(\cof\,\nabla \boldsymbol{y}(t))$ in $\mathcal{M}_{\rm b}(\Omega;\R^{3 \times 3 \times 3})$.}
    \end{split}
\end{equation}
The candidate solution $\boldsymbol{q}\colon [0,T]\to \widetilde{\mathcal{Q}}$ is pointwise defined by this procedure.
Note that, by \eqref{eqn:helly-2}, there holds $\boldsymbol{z}(t)=\mathcal{Z}(\boldsymbol{q}(t))$.
{\MMM Also, we choose the subsequence in \eqref{eqn:t-dependent-subsequence} in order to have $\vartheta_{n_{k_\ell^t}}(t)\to \bar{\vartheta}(t)$.}

\textbf{Step 3 (Stability of the limiting function).} We claim that $\boldsymbol{q}$ satisfies \eqref{eqn:energetic-solution-stability}. Fix $t \in [0,T]$.
Henceforth, for simplicity, we will replace the subscripts $n_k$ and $n_{k_\ell^t}$ by $k$ and $k_\ell^t$, respectively. For every $k \in \N$, set $\tau_k(t)\coloneqq \max \{s \in \Pi_k:\:s \leq t\}$ and note that $\tau_k(t)\to t$, since $|\Pi_k|\to 0$. Then, $\boldsymbol{q}_k(t)=\boldsymbol{q}_k(\tau_k(t))$ so that, by \eqref{eqn:imp-stability-regularization}, we have
\begin{equation}
    \label{eqn:stability-t-1}
    \forall\, \widehat{\boldsymbol{q}}\in \widetilde{\mathcal{Q}},\quad \widetilde{\mathcal{E}}(\tau_k(t),\boldsymbol{q}_k(t))\leq \widetilde{\mathcal{E}}(\tau_k(t),\widehat{\boldsymbol{q}})+\mathcal{D}(\boldsymbol{q}_k(t),\widehat{\boldsymbol{q}}).
\end{equation}
Recall \eqref{eqn:t-dependent-subsequence}. Arguing as in the proof of Proposition \ref{prop:solutions-imp-regularization} and exploiting the continuity of the applied loads in \eqref{eqn:applied-loads-regularity}, we obtain
\begin{equation}
    \label{eqn:stability-t-2}
    \widetilde{\mathcal{E}}(t,\boldsymbol{q}(t))\leq \liminf_{\ell \to \infty}  \widetilde{\mathcal{E}}(\tau_{k_\ell^t}(t),\boldsymbol{q}_{k_\ell^t}(t)).
\end{equation}
Moreover, by the continuity of the applied loads in \eqref{eqn:applied-loads-regularity}, there holds
\begin{equation}
    \label{eqn:stability-t-3}
    \forall\,\widehat{\boldsymbol{q}}\in \widetilde{\mathcal{Q}},\quad \widetilde{\mathcal{E}}(\tau_{k_\ell^t}(t),\widehat{\boldsymbol{q}})\to \widetilde{\mathcal{E}}(t,\widehat{\boldsymbol{q}}),
\end{equation}
while, as $\boldsymbol{z}_{k_\ell^t}(t)\to \boldsymbol{z}(t)$ in $L^1(\Omega;\R^3)$ and $\boldsymbol{z}(t)=\mathcal{Z}(\boldsymbol{q}(t))$, we have 
\begin{equation}
    \label{eqn:stability-t-4}
    \forall\,\widehat{\boldsymbol{q}}\in \widetilde{\mathcal{Q}},\quad \mathcal{D}(\boldsymbol{q}_{k_\ell^t}(t),\widehat{\boldsymbol{q}})=||\mathcal{Z}(\boldsymbol{q}_{k_\ell^t}(t))-\mathcal{Z}(\widehat{\boldsymbol{q}})||_{L^1(\Omega;\R^3)}\to ||\mathcal{Z}(\boldsymbol{q}(t))-\mathcal{Z}(\widehat{\boldsymbol{q}})||_{L^1(\Omega;\R^3)}=\mathcal{D}(\boldsymbol{q}(t),\widehat{\boldsymbol{q}}).
\end{equation}

Hence, combining \eqref{eqn:stability-t-1}--\eqref{eqn:stability-t-4}, we deduce
\begin{equation*}
    \begin{split}
        \widetilde{\mathcal{E}}(t,\boldsymbol{q}(t))&\leq \liminf_{\ell \to \infty}  \widetilde{\mathcal{E}}(\tau_{k_\ell^t}(t),\boldsymbol{q}_{k_\ell^t}(t))\\
        &\leq \liminf_{\ell \to \infty} \left \{\widetilde{\mathcal{E}}(\tau_{k_\ell^t}(t),\widehat{\boldsymbol{q}})+\mathcal{D}(\boldsymbol{q}_{k_\ell^t}(t),\widehat{\boldsymbol{q}}) \right \}\\
        &=\widetilde{\mathcal{E}}(t,\widehat{\boldsymbol{q}})+\mathcal{D}(\boldsymbol{q}(t),\widehat{\boldsymbol{q}}),
    \end{split}
\end{equation*}
for every $\widehat{\boldsymbol{q}}\in \widetilde{\mathcal{Q}}$, which gives \eqref{eqn:energetic-solution-stability} for $t$ fixed.
\end{proof}

\textbf{Step 4  (Upper energy estimate).} We claim that $\boldsymbol{q}$ satisfies the upper energy estimate
\begin{equation}
    \label{eqn:upper-energy-estimate}
    \forall t \in [0,T], \quad  \widetilde{\mathcal{E}}(t,\boldsymbol{q}(t))+\var_{\mathcal{D}}(\boldsymbol{q};[0,t])\leq \widetilde{\mathcal{E}}(0,\boldsymbol{q}^0)+\int_0^t \partial_t \widetilde{\mathcal{E}}(\tau,\boldsymbol{q}(\tau))\,\d\tau.
\end{equation}
Recall \eqref{eqn:a-priori-bound}. For every $n \in \N$, using \eqref{eqn:total-energy-regularized-gronwall}, we obtain
\begin{equation}
    \label{eqn:gronwall-modulus-continuity}
    \forall s,t\in[0,T], \quad |\widetilde{\mathcal{E}}(t,\boldsymbol{q}_n(t))-\widetilde{\mathcal{E}}(s,\boldsymbol{q}_n(s))|\leq (C+M) \left |e^{L|t-s|}-1 \right |\eqqcolon \rho(t-s), 
\end{equation}
where $\rho(r)\to 0$, as $r \to 0$. 

Fix $t \in [0,T]$, so that $\boldsymbol{q}_n(t)=\boldsymbol{q}_n(\tau_n(t))$ and $\var_{\mathcal{D}}(\boldsymbol{q}_n;[0,t])=\var_{\mathcal{D}}(\boldsymbol{q}_n;[0,\tau_n(t)])$ for every $n \in \N$. Recall the definition of $\vartheta_n$ in Step 2. By \eqref{eqn:gronwall-modulus-continuity} and \eqref{eqn:pc-interpolants-energy-inequality-n}, we have
\begin{equation}
    \label{eqn:upper-energy-estimate-1}
    \begin{split}
        \widetilde{\mathcal{E}}(t,\boldsymbol{q}_n(t))+\var_{\mathcal{D}}(\boldsymbol{q}_n;[0,t])&\leq \widetilde{\mathcal{E}}(\tau_n(t),\boldsymbol{q}_n(\tau_n(t)))+\var_{\mathcal{D}}(\boldsymbol{q}_n;[0,\tau_n(t)])+\rho(|\Pi_n|)\\
        &\leq \widetilde{\mathcal{E}}(0,\boldsymbol{q}^0)+\int_0^{\tau_n(t)} \vartheta_n(\tau)\,\d\tau+\rho(|\Pi_n|),
    \end{split}
\end{equation}
for every $n \in \N$.
Also, by the lower semicontinuity of the total variation, we have
\begin{equation}
    \label{eqn:upper-energy-estimate-2}
    \var_{\mathcal{D}}(\boldsymbol{q};[0,t])=\var(\boldsymbol{z};[0,t])\leq \liminf_{n \to \infty} \var(\boldsymbol{z}_n;[0,t])=\liminf_{n \to \infty} \var_{\mathcal{D}}(\boldsymbol{q}_n;[0,t]),
\end{equation}
as \eqref{eqn:helly-2} holds and $\boldsymbol{z}(s)=\mathcal{Z}(\boldsymbol{q}(s))$ for every $s \in [0,T]$.
Then, from \eqref{eqn:stability-t-2} and \eqref{eqn:upper-energy-estimate-1}--\eqref{eqn:upper-energy-estimate-2}, we deduce
\begin{equation}
\label{eqn:upper-energy-estimate-3}
    \begin{split}
        \widetilde{\mathcal{E}}(t,\boldsymbol{q}(t))+\var_{\mathcal{D}}(\boldsymbol{q};[0,t])&\leq \liminf_{\ell \to \infty} \left \{ \widetilde{\mathcal{E}}(t,\boldsymbol{q}_{k_\ell^t}(t))+ \var_{\mathcal{D}}(\boldsymbol{q}_{k_\ell^t};[0,t]) \right \}\\
        &\leq \widetilde{\mathcal{E}}(0,\boldsymbol{q}^0)+\liminf_{\ell \to \infty} \left \{ \int_0^{\tau_{k_\ell^t}(t)} \vartheta_{k_\ell^t}(\tau)\,\d \tau + \rho(|\Pi_{k_\ell^t}|)  \right\}\\
        &=\widetilde{\mathcal{E}}(0,\boldsymbol{q}^0)+\int_0^t \vartheta(\tau)\,\d\tau\leq \widetilde{\mathcal{E}}(0,\boldsymbol{q}^0)+\int_0^t \bar{\vartheta}(\tau)\,\d\tau,
    \end{split}
\end{equation}
where, in the last line, we used that $\vartheta_{k}\wks \vartheta$ in $L^\infty(0,T)$, $\vartheta \leq \bar{\vartheta}$ and $\rho(|\Pi_{k}|)\to 0$.

We claim that $\bar{\vartheta}(t)=\partial_t \widetilde{\mathcal{E}}(t,\boldsymbol{q}(t))$ for every $t \in (0,T)$. Fix $t \in (0,T)$. Testing \eqref{eqn:energetic-solution-stability} with 
$\boldsymbol{q}_{k_\ell^t}(t)$, we have
\begin{equation*}
    -\widetilde{\mathcal{E}}(t,\boldsymbol{q}_{k_\ell^t}(t))\leq - \widetilde{\mathcal{E}}(t,\boldsymbol{q}(t))+\mathcal{D}(\boldsymbol{q}(t),\boldsymbol{q}_{k_\ell^t}(t))
\end{equation*}
so that, using \eqref{eqn:stability-t-4}, we compute
\begin{equation*}
    \begin{split}
        \limsup_{\ell \to \infty} \widetilde{\mathcal{E}}(t,\boldsymbol{q}_{k_\ell^t}(t))&=-\liminf_{\ell \to \infty} \left (-\widetilde{\mathcal{E}}(t,\boldsymbol{q}_{k_\ell^t}(t))\right )\\
        &\leq -\liminf_{\ell \to \infty} \left (- \widetilde{\mathcal{E}}(t,\boldsymbol{q}(t))+\mathcal{D}(\boldsymbol{q}(t),\boldsymbol{q}_{k_\ell^t}(t)) \right )\leq \widetilde{\mathcal{E}}(t,\boldsymbol{q}(t)).
    \end{split}
\end{equation*}
Given \eqref{eqn:stability-t-2}, we conclude that $\widetilde{\mathcal{E}}(t,\boldsymbol{q}_{k_\ell^t}(t))\to \widetilde{\mathcal{E}}(t,\boldsymbol{q}(t))$. Recalling \eqref{eqn:t-dependent-subsequence}, by \cite[Proposition 5.6]{mielke}, we have $\vartheta_{k_\ell^t}(t)=\partial_t \widetilde{\mathcal{E}}(t,\boldsymbol{q}_{k_\ell^t}(t))\to \partial_t \widetilde{\mathcal{E}}(t,\boldsymbol{q}(t))$ and, as $\vartheta_{k_\ell^t}(t)\to\bar{\vartheta}(t)$, we deduce $\bar{\vartheta}(t)=\partial_t \widetilde{\mathcal{E}}(t,\boldsymbol{q}(t))$. Therefore, 
\eqref{eqn:upper-energy-estimate-3} gives \eqref{eqn:upper-energy-estimate} for fixed $t$.

\textbf{Step 5 (Lower energy estimate).} Finally, we show that $\boldsymbol{q}$ satisfies
\begin{equation}
    \label{eqn:lower-energy-estimate}
    \forall t \in [0,T], \quad  \widetilde{\mathcal{E}}(t,\boldsymbol{q}(t))+\var_{\mathcal{D}}(\boldsymbol{q};[0,t])\geq \widetilde{\mathcal{E}}(0,\boldsymbol{q}^0)+\int_0^t \partial_t \widetilde{\mathcal{E}}(\tau,\boldsymbol{q}(\tau))\,\d\tau,
\end{equation}
which, combined with \eqref{eqn:upper-energy-estimate}, proves \eqref{eqn:energetic-solution-energy-balance}.

Note that, by \eqref{eqn:stability-t-2}, we have $\sup_{t \in [0,T]}\widetilde{\mathcal{E}}(t,\boldsymbol{q}(t))\leq C$. Moreover, the function is $t \mapsto \partial_t \widetilde{\mathcal{E}}(t,\boldsymbol{q}(t))$ belongs to $L^\infty(0,T)$, as it coincides with $\bar{\vartheta}$. Hence, by \cite[Proposition 5.7]{mielke}, for every $s,t \in [0,T]$ with $s<t$ we have
\begin{equation*}
    \widetilde{\mathcal{E}}(t,\boldsymbol{q}(t))+\var_{\mathcal{D}}(\boldsymbol{q};[s,t])\geq \widetilde{\mathcal{E}}(s,\boldsymbol{q}(s))+\int_s^t \partial_t \widetilde{\mathcal{E}}(\tau,\boldsymbol{q}(\tau))\,\d\tau,
\end{equation*}
which in turn yields \eqref{eqn:lower-energy-estimate}.

\section*{Acknowledgements}
The authors acknowledge support
from the Austrian Science Fund (FWF) projects F65, V 662, Y1292,
from the FWF-GA\v{C}R project I 4052/19-29646L, and  
from the OeAD-WTZ project CZ04/2019 (M\v{S}MT\v{C}R 8J19AT013). MK  also thanks   the Center of Advanced Applied Sciences (CAAS), financially supported by the European Regional Development Fund (Project no.
CZ.$02.1.01/0.0/0.0/16_019/0 0 0 0778$.


\begin{thebibliography}{50}

{\MMM
\bibitem{antman.rogers}
S. S. Antman, R. C. Rogers.
\newblock{Steady-state problems of nonlinear electro-magneto-thermoelasticity.}
\newblock{\em Arch. Ration. Mech. Anal.} \textbf{95} (1986), 279--323.
}

\bibitem{ball}
J. M. Ball.
\newblock{Global invertibility of Sobolev functions and the interpenetration of matter.}
\newblock{\em Proc. Roy. Soc. Edinburgh
Sect. A} \textbf{88} (1981), 315--328.

\bibitem{ball.currie.olver}
J. M. Ball, J. C. Currie, P. J. Olver.
\newblock{Null Lagrangians, weak continuity, and variational problems of arbitrary order.}
\newblock{\em J. Funct. Anal.} \textbf{41} (1981), 135--174.


\bibitem{barchiesi.desimone}
M. Barchiesi, A. De Simone.
\newblock{Frank energy for nematic elastomers: a nonlinear model.}
\newblock{\em ESAIM Control Optim. Calc.
Var.} \textbf{2} (2015), 372--377.

\bibitem{barchiesi.henao.moracorral}
M. Barchiesi, D. Henao, C. Mora-Corral.
\newblock{Local Invertibility in Sobolev spaces with applications to nematic elastomers and magnetoelasticity.}
\newblock{\em Arch. Ration. Mech. Anal.} \textbf{224} (2017), 743--816.

\bibitem{benesova.kruzik.schloemerkemper}
B. Bene\v{s}ov\'{a}, M. Kru\v{z}\'{\i}k, A. Schl\"{o}merkemper.
\newblock{A note on locking materials and gradient polyconvexity.}
\newblock{\em Math. Mod. Meth. Appl. Sci.} \textbf{28} (2018), 2367--2401.

{\MMM 
\bibitem{bouchala.hencl.molchanova}
O. Bouchala, S. Hencl, A. Molchanova. Injectivity almost everywhere for weak limits of Sobolev homeomorphisms. {\em J.
Funct. Anal.} \textbf{279} (2020), 108658.
}


\bibitem{bresciani}
M. Bresciani. Linearized von K\'arm\'an theory for incompressible magnetoelastic plates. {\MMM {\em Math. Mod. Meth. Appl. Sci.} \textbf{31} (2021), 1987--2037.}

\bibitem{brown}
W. F. Brown. {\em Magnetoelastic Interactions}. Springer, Berlin, 1966.


\bibitem{ciarlet.necas}
P. G. Ciarlet, J. Ne\v{c}as.
\newblock{Injectivity and self-contact in nonlinear elasticity.}
\newblock{\em Arch. Ration. Mech. Anal.} \textbf{97} (1987), 173--188.





\bibitem{davoli.difratta}
E. Davoli, G. Di Fratta. Homogenization of chiral magnetic materials. - A mathematical evidence of Dzyaloshinskii's predictions on helical structures.
{\em J. Nonlinear Sci.} \textbf{30} (2020), 1229--1262.



\bibitem{davoli.kruzik.piovano.stefanelli}
E. Davoli, M. Kru\v{z}\'{\i}k, P. Piovano, U. Stefanelli.
\newblock{Magnetoelastic thin films at large strains}.
\newblock{\em Cont. Mech.  Thermodyn.} \textbf{33} (2021), 327--341.


\bibitem{desimone.dolzmann}
A. DeSimone, G. Dolzmann. Existence of minimizers for a variational problem in
two-dimensional nonlinear magnetoelasticity. {\em Arch. Ration. Mech. Anal}. \textbf{144} (1998), 107--120. 

\bibitem{desimone.james} A. DeSimone, R. D. James. A constrained theory of magnetoelasticity. {\em J. Mech. Phys.
Solids}, \textbf{50} (2002), 283–320.

{\MMM
\bibitem{desimone.podioguidugli}
A. DeSimone, P. Podio-Guidugli.
\newblock{On the continuum theory of deformable ferromagnetic solids.}
\newblock{\em Arch. Ration. Mech. Anal.} \textbf{136} (1996), 201--233.
}


\bibitem{dzyaloshinsky}
I. Dzyaloshinsky. A thermodynamic theory of ``weak" ferromagnetism of antiferromagnetics. {\em J. Phys. Chem. Solids}, \textbf{4} (1958), 241--255.

{\MMM
\bibitem{eisen}
G. Eisen. A selection lemma for sequences of measurable sets, and lower semicontinuity of multiple integrals. {\em Manuscripta Math.} \textbf{27} (1979), 73–-79.
}

{\MMM
\bibitem{federer}
H. Federer.
\newblock{\em Geometric Measure Theory.}
Springer, New York, 1969.
}


\bibitem{fonseca.leoni}
I. Fonseca, G. Leoni.
\newblock{ \em Modern Methods in the Calculus of Variations: $L^p$ spaces}.
\newblock{Springer, New York, 2007.}

{\MMM
\bibitem{fonseca.gangbo}
I. Fonseca, W. Gangbo.
\newblock{Local invertibility of Sobolev functions.}
\newblock{\em SIAM J. Math. Anal.} \textbf{26} (1995), 280--304.

\bibitem{fonseca.gangbo2}
I. Fonseca, W. Gangbo.
\newblock{\em Degree Theory in Analysis and Applications.}
Oxford University Press, New York, 1995.
}


\bibitem{francfort.mielke}
G. Francfort, A. Mielke.
\newblock{Existence results for a class of rate-independent material models with nonconvex elastic energies.}
\newblock{\em J. Reine Angew. Math.} \textbf{595} (2004), 55--91.



\bibitem{grandi.kruzik.mainini.stefanelli}
D. Grandi, M. Kru\v{z}\'ik, E. Mainini, U. Stefanelli. A phase-field approach to Eulerian interfacial energies.
{\em Arch. Ration. Mech. Anal.} \textbf{234} (2019), 351--373.

\bibitem{hajlasz}
P. Haj\l{}asz. 
\newblock{Change of variables formula under minimal assumptions.}
\newblock{\em Colloq. Math.} \textbf{64} no. 1 (1993). 

\bibitem{henao.moracorral}
D. Henao, C. Mora-Corral.
\newblock{Lusin’s condition and the distributional determinant for
deformations with finite energy.}
\newblock{\em Adv. Calc. Var.} \textbf{5} (2012), 355--409.

\bibitem{henao.moracorral2}
D. Henao, C. Mora-Corral.
\newblock{Regularity of inverses of Sobolev deformations with
finite surface energy.}
\newblock{\em J. Funct. Anal.} \textbf{268} (2015), 2356--2378.

\bibitem{henao.moracorral.oliva}
D. Henao, C. Mora-Corral, M. Oliva.
\newblock{Global invertibility of Sobolev maps.}
\newblock{\em Adv. Calc. Var.} \textbf{14} (2021), 207--230.

{\MMM
\bibitem{henao.stroffolini}
D. Henao, B. Stroffolini.
\newblock{Orlicz-Sobolev nematic elastomers.}
\newblock{\em Nonlinear Anal.} \textbf{194} (2020), 111513.
}



\bibitem{james.kinderlehrer} R. D. James, D. Kinderlehrer. Theory of magnetostriction with applications to $Tb_x Dy_{1-x}Fe_2$ {\em Philos. Mag. B}. \textbf{68} (1993), 237--274.

{\MMM
\bibitem{kankanala.triantafyllidis}
S. V. Kankanala, N. Triantafyllidis.
\newblock{On finitely strained magnetorheological elastomers.}
\newblock{\em J. Mech. Phys. Solids} \textbf{52} (2004), 2869--2908.
}

\bibitem{kromer}
S. Kr\"{o}mer.
\newblock{Global invertibility for orientation-preserving Sobolev maps via invertibility on or near the boundary.}
\newblock{\em Arch. Ration. Mech. Anal.} \textbf{238} (2020), 1113--1155.

\bibitem{kruzik.roubicek}
M. Kru\v{z}\'{\i}k, T. Roub\'{\i}\v{c}ek.
\newblock{ \em Mathematical Methods in Continuum Mechanics of Solids.}
\newblock{Springer, Cham, 2019.}

\bibitem{kruzik.stefanelli.zanini}
M. Kru\v{z}\'{\i}k, U. Stefanelli, C. Zanini.
\newblock{  Quasistatic evolution of magnetoelastic plates via dimension reduction.}
\newblock{\em Discrete Contin. Dyn. Syst.} \textbf{35} (2015), 2615--2623.



\bibitem{kruzik.stefanelli.zeman}
M. Kru\v{z}\'{\i}k, U. Stefanelli, J. Zeman.
\newblock{Existence results for incompressible magnetoelasticity.}
\newblock{\em Discrete Contin. Dyn. Syst.} \textbf{35} (2015), 5999--6013.

\bibitem{li.melcher} X. Li, C. Melcher. Stability of axisymmetric chiral skyrmions. {\em J. Funct. Anal.} \textbf{275} (2018), 2817--2844.

\bibitem{liakhova}
J. Liakhova.
\newblock{ A theory of magnetostrictive thin films with applications.}
\newblock{Ph.D. Thesis, University of Minnesota, 1999.}

\bibitem{liakhova.luskin.zhang}
J. Liakova, M. Luskin, T. Zhang.
\newblock{  Computational modeling of ferromagnetic shape memory thin films.}
\newblock{\em Ferroelectrics.} \textbf{342} (2006), 7382.

\bibitem{lund.muratov}
R. G. Lund, C. B. Muratov. One-dimensional domain walls in thin ferromagnetic films with fourfold anisotropy. {\em Nonlinearity}, \textbf{29} (2016), 1716--1734.

\bibitem{luskin.zhang}
M. Luskin, T. Zhang.
\newblock{ Numerical analysis of a model for ferromagnetic shape memory thin films}.
\newblock{\em Comput. Methods Appl. Mech. Engrg.} \textbf{196} (2007), 37--40.

{\MMM
\bibitem{mainik.mielke05}
A. Mainik, A. Mielke. Existence results for energetic models for rate-independent systems. {\em Calc.
Var.} \textbf{22} (2005), 73--99.}

\bibitem{mainik.mielke}
A. Mainik, A. Mielke.
\newblock{Global Existence for Rate-Independent Gradient Plasticity at Finite Strain.}
\newblock{\em J. Nonlinear Sci.} \textbf{19} (2009), 221--248.


\bibitem{marcus.mizel}
M. Marcus, V. J. Mizel.
\newblock{Transformations by functions in Sobolev spaces and lower semicontinuity for parametric variational problems.}
\newblock{\em Bull. Am. Math. Soc.} \textbf{79} (1973), 790--795.

\bibitem{melcher}
C. Melcher. Chiral skyrmions in the plane. {\em  Proc. R. Soc. Lond. Ser. A Math. Phys. Eng. Sci.} \textbf{470} (2014), p. 20140394. 

\bibitem{mielke}
A. Mielke.
\newblock{Evolution of rate-independent systems}. In C. M. Dafermos, E. Feireisl.
\newblock{\em Handbook of Differential Equations. Volume II. Evolutionary Equations.}
\newblock{Elsevier, Amsterdam, 2005.}

\bibitem{mielke.roubicek}
A. Mielke, T. Roub\'{i}\v{c}ek.
\newblock{ \em Rate-independent Systems. Theory and Application.}
\newblock{Springer, New York, 2015.}

\bibitem{moriya}
T. Moriya. Anisotropic superexchange interaction and weak ferromagnetism. \textit{Physical Review} \textbf{120} (1960), 91. 

\bibitem{mueller.qi.yan}
S. M\"{u}ller, T. Qi, B. S. Yan.
\newblock{On a new class of elastic deformations not allowing for cavitation.}
\newblock{\em Ann. Inst. Henri Poincaré Anal. Non Linéaire} \textbf{11} (1994) 217--243.

\bibitem{mueller.spector}
S. M\"{u}ller, S. J. Spector.
\newblock{An existence theory for nonlinear elasticity
that allows for cavitation.}
\newblock{\em Arch. Ration. Mech. Anal.}, \textbf{131} (1995), 1--66.

\bibitem{muratov.slastikov}
C. B. Muratov, V. V. Slastikov. Domain structure of ultrathin ferromagnetic elements in the presence of Dzyaloshinskii-Moriya interaction. {\em  Proc. R. Soc. Lond. Ser. A Math. Phys. Eng. Sci.} \textbf{473} (2017), 20160666. 



{\MMM
\bibitem{rogers}
R. C. Rogers. 
\newblock{Nonlocal variational problems in nonlinear electromagneto-elastostatics.}
\newblock{\em SIAM J. Math. Anal.}  \textbf{19} (1988), 1329-–1347.

}

\bibitem{roubicek.tomassetti}
T. Roub\'{i}\v{c}ek, G. Tomassetti.
\newblock{ A thermodynamically consistent model of magneto-elastic materials under diffusion at large strains and its analysis.}
\newblock{\em Z. Angew. Math. Phys.} \textbf{69} (2018), Article number: 55. 



\bibitem{rybka.luskin}
P. Rybka, M. Luskin. Existence of energy minimizers for magnetostrictive materials.
{\em SIAM J. Math. Anal.} \textbf{36} (2005), 2004--2019.



\bibitem{qi}
T. Qi.
\newblock{Almost-everywhere injectivity in nonlinear elasticity.}
\newblock{\em Proc. Roy. Soc.
Edinburgh Sect. A}, \textbf{109} (1988), 79--95.

{\MMM
\bibitem{sharma.saxena}
B. L. Sharma, P. Saxena.
\newblock{Variational principles of nonlinear magnetoelastostatics and their correspondences.}
\newblock{\em Math. Mech. Solids} \textbf{26} (2021), 1424--1454.
}

\bibitem{sverak}
V. \v{S}ver\'{a}k.
\newblock{Regularity properties of deformation with finite energy.}
\newblock{\em Arch. Ration. Mech. Anal.} \textbf{100} (1988), 105--127.


\bibitem{toupin1}
R. A. Toupin. 
\newblock{Elastic materials with couple stresses.}
\newblock{\em Arch. Ration. Mech. Anal.} \textbf{11} (1962), 385--414.

\bibitem{toupin2}
R. A. Toupin.
\newblock{Theories of elasticity with couple stress.}
\newblock{\em Arch. Ration. Mech. Anal.} \textbf{17} (1964), 85--112.









\end{thebibliography}
\end{document}